\documentclass[reqno,12pt]{amsart}
\usepackage{amsmath, amssymb, amsthm, amsfonts} 
\usepackage[english]{babel}
\usepackage{bbm}
\usepackage{graphicx}
\usepackage{url}
\usepackage{soul}
\usepackage{epstopdf}
\usepackage[ruled,vlined]{algorithm2e}
\usepackage[a4paper,bindingoffset=0.5cm,left=2cm,right=2cm,top=2.5cm,bottom=2cm,footskip=.8cm]{geometry}
\usepackage{rotating}
\usepackage{amsbsy,enumerate}
\usepackage{comment}
\usepackage{mathrsfs} 
\usepackage{enumitem}
\usepackage{cancel}

\newcommand{\bs}{\boldsymbol}

\newcommand{\vb}{\vspace{3.2mm}}

\newcommand{\vertiii}[1]{{\left\vert\kern-0.25ex\left\vert\kern-0.25ex\left\vert #1 \right\vert\kern-0.25ex\right\vert\kern-0.25ex\right\vert}}
\allowdisplaybreaks
\usepackage{hyperref}

\newcommand{\p}{\mathbb{P}}

\allowdisplaybreaks

\newtheorem{lemma}{Lemma}
\newtheorem{corollary}{Corollary}

\newtheorem{theorem}{Theorem}
\newtheorem{remark}{Remark}

\newtheorem{proposition}{Proposition}

\usepackage[utf8]{inputenc}
\usepackage{type1cm}         
\usepackage{multicol}        
\usepackage[bottom]{footmisc}

\usepackage{newtxtext}       %
\usepackage{newtxmath}       

\begin{document}

	\title[Risk theory in a finite customer-pool setting]{Risk theory in a finite customer-pool setting}
\author{Michel Mandjes {\tiny and} Dani\"el T. Rutgers}

	\begin{abstract}
    This paper investigates an insurance model with a finite number of major clients and a large number of small clients, where the dynamics of the latter group are modeled by a spectrally positive Lévy process. We begin by analyzing this general model, in which the inter-arrival times are exponentially distributed (though not identically), and derive the closed-form Laplace transform of the ruin probability. Next, we examine a simplified version of the model involving only the major clients, and explore the tail asymptotics of the ruin probability, focusing on the cases where the claim sizes follow phase-type or regularly varying distributions. Finally, we derive the distribution of the overshoot over an exponentially distributed initial reserve, expressed in terms of its Laplace-Stieltjes transform.

\vb


\vb

\noindent
{\sc Affiliations.} MM is with the Mathematical Institute, Leiden University, P.O. Box 9512,
2300 RA Leiden,
The Netherlands. He is also affiliated with Korteweg-de Vries Institute for Mathematics, University of Amsterdam, Amsterdam, The Netherlands; E{\sc urandom}, Eindhoven University of Technology, Eindhoven, The Netherlands; Amsterdam Business School, Faculty of Economics and Business, University of Amsterdam, Amsterdam, The Netherlands. 

\noindent
DR is with the Mathematical Institute, Leiden University, P.O. Box 9512, 2300 RA Leiden, The Netherlands.

\noindent Date: {\it \today}.

\vb

\noindent
{\sc Corresponding author.} Dani\"el Rutgers
\vb

\noindent
{\sc Email.} \url{m.r.h.mandjes@math.leidenuniv.nl}, \url{d.t.rutgers@math.leidenuniv.nl}.

	\end{abstract}

 \maketitle

\section*{I. Introduction}
In conventional ruin theory, the typical setting involves an infinitely large pool of independent and essentially homogeneous clients. Within this framework, the {\it Cramér-Lundberg} model \cite{AA, MB} serves as a natural representation of an insurance firm’s surplus process. Starting from an initial surplus level, claims arrive according to a Poisson process, their sizes are modeled as independent and identically distributed random variables, and the firm receives premium payments that increase linearly over time. This cornerstone model, playing a central role in actuarial science, and various of its variants, has been extensively studied. While various metrics are examined in the actuarial literature, the ruin probability of the insurance firm is arguably the most important quantity of interest.

In the present paper we consider the setting in which we distinguish between the insurance firm's (finitely many) `major clients' and its `small clients'. A claim submitted by a major client can be interpreted as its default, as the client effectively exits the system following the submission of the claim. The dynamics of the small clients are modelled via the flexible class of spectrally-positive L\'evy processes \cite{DM, KYP}, i.e., L\'evy processes without negative jumps, covering compound Poisson processes and Brownian motion as special cases. Such L\'evy processes are, thanks to a cental-limit type of argumentation, particularly suited to model the claims submitted by a large aggregate of clients; the contribution of large numbers of clients converge, after an appropriate centering and scaling, in the finite-variance regime to Brownian motion, and in the infinite-variance regime to alpha-stable L\'evy motion \cite{WHI}. 

\medskip

In this paper we analyze several quantities that play a pivotal role in ruin theory. The first key object that we study is the insurance firm's ruin probability, defined as the probability that, with an initial surplus $u>0$, the surplus level falls below 0. We distinguish two cases: ruin occurring before a predefined time horizon, and ruin occurring over an infinite time horizon (in the latter case, assuming the net profit condition is satisfied). In the present work we succeed in uniquely characterizing these objects in terms of their associated Laplace transform (with respect to $u$, that is); numerical algorithms are available to invert these transforms \cite{AW, dI}. In the second place, relying on these results for the Laplace transform of the ruin probability, we provide expressions for the asymptotics of the ruin probability, valid in the regime that the initial reserve level $u$ grows large. We do so for the claim sizes being of phase type, which form a class of light-tailed distributions by which any random variable on the positive half line can be approximated arbitrarily closely, as well as for heavy-tailed claim sizes of the regularly varying type. The third contribution concerns an analysis of the so-called `overshoot', which refers to the magnitude by which the surplus level drops below $0$ in case of ruin. 

\medskip

A model that is intimately related to the model studied in the present paper, is the {\it queueing} system with finitely many customers. In the most basic variant of this model, the customer's {\it arrival times} (rather than their interarrival times) are independent, identically distributed random variables. Early contributions from the queueing literature are \cite{GH, JS}; see also the survey \cite{HL}. In \cite{BET,BHL,HJW} the focus lies on the development of  diffusion approximations that are valid under a heavy-traffic scaling, whereas \cite{HON} performs a rare-event analysis by establishing a sample-path large deviations principle. Whereas these diffusion and large deviations results provide insight into the system's probabilistic behavior in specific asymptotic regimes, the recent papers \cite{BKM, MR}  present exact results, albeit in terms of transforms, the former on the workload and the latter on the number of customers; see, however, also the recent results in \cite{YOS}. 

Results for models similar to our ruin model can be found in e.g.\ \cite{DM, KMD}; it is observed that the net cumulative claim process in our setup can be interpreted as a specific spectrally-positive Markov additive process with a non-irreducible background process. There is a vast literature on extreme values attained by spectrally one-sided Markov additive processes, illustrative examples being \cite{DIKM,IBM}. In this context it is noted that in our paper we succeed in obtaining relatively explicit results, by exploiting the specific structure underlying our model; at an abstract level, we use that the transition matrix of the background process has a special upper triangular structure in which only the entries directly above the diagonal are positive.

\medskip

This paper is organized as follows. In Section II, we provide a detailed description of the model and outline the objectives of the paper. Section~III presents an exact analysis of the Laplace-Stieltjes transform of the running maximum of the net cumulative claim process, which is then used to derive the Laplace transform of the ruin probability. Sections~IV and~V discuss the tail asymptotics of the ruin probability for phase-type and regularly varying claim sizes, respectively. Section~VI offers a comprehensive analysis of the overshoot at ruin. Throughout the paper, our results are illustrated with numerical examples.

\section*{II. Model \& objectives}
In this work we consider an insurance model with a finite number of $m\in{\mathbb N}$ `large claims', which are due to `major clients', and a large aggregate of `smaller clients'.

\subsubsection*{Major clients.} In the basic variant of our arrival process, the arrival times of the claims by major clients, denoted by $A_1,\ldots,A_m$, are independent exponentially distributed random variables with parameter $\lambda>0$ (and hence mean $1/\lambda$). Let $A_{(1)} \leqslant A_{(2)} \leqslant \dots \leqslant A_{(m)}$ denote the corresponding order statistics, i.e., the arrival times in increasing order, adopting the convention that $A_{(0)} := 0$ and $A_{(m+1)} := \infty$. Recall that $A_{(1)}$, being the minimum of $m$ exponential random variables with parameter $\lambda$, is exponentially distributed with parameter $m\lambda$; more generally, the time between $A_{(n)}$ and $A_{(n+1)}$, for $n=0,..,m-1$, is exponentially distributed with parameter $\lambda(m-n)$. 

We can, however, deal with the slightly more general variant in which the time between $A_{(n)}$ and $A_{(n+1)}$ is exponentially distributed with parameter $\lambda^\circ_{m-n}$, for positive ${\bs\lambda}\equiv (\lambda_1,\ldots,\lambda_m)$.
In our analysis we consider ruin during an exponentially distributed time interval with mean $\beta^{-1}$ for some $\beta>0$. In the sequel we use the compact notation $\lambda_n\equiv\lambda_n(\beta):=\lambda^\circ_n+\beta$ for $n\in\{1,\ldots,m\}$.

The claim of the $n$-th arriving major client is distributed as a generic non-negative random variable $B_n$; the vector ${\bs B}\equiv (B_1,\ldots,B_m)$ consists of independent entries that are independent of the claim arrival times. The claim size of the $n$-th arriving major client is characterized via its Laplace-Stieltjes transform (LST): for $\alpha\geqslant 0$,
\[{\mathscr B}_n(\alpha):={\mathbb E}\,e^{-\alpha B_n};\]
we denote ${\bs {\mathscr B}}(\cdot)\equiv ({\mathscr B}_1(\cdot), \ldots, {\mathscr B}_m(\cdot)).$

\subsubsection*{Small clients.} 
We model the net cumulative claim process corresponding to the small clients by a spectrally-positive L\'evy process \cite{KYP}. We allow this spectrally-positive L\'evy process to depend on the number of large clients that are still `alive', i.e., that have not yet submitted a claim; this entails that we are able to model the effect of claims by large clients on the net cumulative claim process of smaller clients. The class of spectrally-positive L\'evy processes contains Brownian motion, which is, by virtue of the functional central limit theorem, a natural candidate to model the net cumulative claim process of relatively small clients. 

\noindent
Let $N(t)\in\{0,\ldots,m\}$ be the number of claims that have {\it not} been submitted by time $t$. For times $t$ at which $N(t) = n$, with $n\in\{0,\ldots,m\}$, we let $Z_n(t)$ denote the spectrally-positive L\'evy process that is `active', characterized via its Laplace exponent \cite[\S 2.7.1]{KYP}
\[\varphi_n(\alpha) =  \log \mathbb{E}\,e^{-\alpha Z_n(1)}= r_n\alpha + \frac{\sigma_n^2\alpha^2}{2}+ \int_0^\infty \left(1-e^{-\alpha x}+ \alpha x {\bs 1}_{\{x\leqslant 1\}}\right)\Pi_n({\rm d}x),
\]
with $r_n\in{\mathbb R}$, $\sigma_n\geqslant 0$, the measure $\Pi_n(\cdot)$ such that 
\[\int_0^\infty \min\{1,x^2\}\,\Pi_n({\rm d}x) <\infty,\]
and with $\psi_n(\cdot)$ denoting its (right-)inverse. In the expression for $\varphi_n(\cdot)$, the first term corresponds to the deterministic drift, the second term to a Brownian motion, and the third to the (positive) jumps.

\subsubsection*{Net cumulative claim process.} The net cumulative claim process $Y(\cdot)$ is defined as the cumulative claim process minus the premiums earned. As such, the ruin probability is the probability that $Y(\cdot)$ exceeds the initial reserve, say $u$.

Using the objects defined above, we can now define the net cumulative claim process pertaining to our model with major clients and small clients. For any $t\in[A_{(n)}, A_{(n+1)})$ and $n\in\{0,\dots,m\}$, we have
\begin{equation}
    Y(t) = Y(A_{(n)}) + Z_n(t) - Z_n(A_{(n)}),
\end{equation}
where $Y(0) = 0$ and $Y(A_{(n)}) = Y(A_{(n)}-) + B_n$ for $n\in\{1,\dots,m\}$.
Note that the primitives of this model consist of the number of major clients $m$, the claim arrival rates ${\bs \lambda}$, the vector-valued LST ${\bs {\mathscr B}}(\cdot)$, and the vector of Laplace exponents ${\bs \varphi}(\cdot)\equiv(\varphi_0(\cdot),\ldots, \varphi_m(\cdot))$. We write
\[Y(\cdot)\in {\mathscr L} [m, {\bs \lambda}, {\bs {\mathscr B}}(\cdot),{\bs \varphi}(\cdot)]. \]
As will become clear in our analysis, the case that, for some $n\in\{0,\ldots,m\}$, $\varphi_n(\cdot)$ corresponds to an increasing subordinator, has to be dealt with separately.  

In some parts of this paper, we specifically deal with the subclass of the spectrally-positive L\'evy processes that has neither a Brownian term nor jumps. More precisely, in that variant we have that $\sigma_n^2=0$ and $\Pi_n(\cdot)=0$ for any $n\in\{1,\dots,m\}$, so that $\varphi_n(\alpha) = r_n\alpha$; i.e., the L\'evy process is a deterministic drift. With ${\bs r} \equiv (r_0,\ldots,r_m)$, we then write, with mild abuse of notation,
\[Y(\cdot)\in {\mathscr L} [m, {\bs \lambda}, {\bs {\mathscr B}}(\cdot),{\bs r}]; \]
in line with our earlier remark concerning the special role of increasing subordinators, we have to distinguish between the case that $r_n$ is positive and that it is negative.

\subsubsection*{Objectives.} 
Define the {\it running maximum process} $\bar Y(\cdot)$ via $\bar Y(t):=\sup\{Y(s):s\in[0,t]\}.$
A main objective of this paper is to identify, with $T_\beta$ denoting an exponentially distributed random variable with parameter $\beta>0$ that is sampled independently from the process $Y(t)$, the LST of the running maximum over an interval of length $T_\beta$:
\[\pi_n(\alpha,\beta):= {\mathbb E}_ne^{-\alpha \bar Y(T_\beta)},\]
the subscript $n$ in ${\mathbb E}_n(\cdot)$ indicating the initial condition that $N(0)=n$, for $n\in\{0,\ldots,m\}$. We do this for the case of the small clients being modeled by general spectrally-positive L\'evy processes, i.e., $Y(\cdot)\in {\mathscr L} [m, {\bs \lambda}, {\bs {\mathscr B}}(\cdot),{\bs \varphi}(\cdot)]$. It is noted that $\pi_n(\alpha,\beta)$ can be interpreted as
\[ \int_0^\infty\int_0^\infty\beta e^{-\alpha x-\beta t}\,{\mathbb P}_n(\bar Y(t)\in {\rm d}x)\,{\rm d}t,\]
so that by applying double Laplace inversion the density ${\mathbb P}_n(\bar Y(t)\in {\rm d}x)$ is obtained. A computational technique for double Laplace inversion has been provided in \cite{dI}, and its efficacy in the context of L\'evy-driven models has been demonstrated in \cite{AdIM}.

Our second objective concerns the analysis of the tail asymptotics of the ruin probabilities
\begin{align*}p_m(u,\beta)&:= {\mathbb P}_m(\bar Y(T_\beta)>u)={\mathbb P}_m(\tau(u)\leqslant T_\beta),\\p_m(u)&:= {\mathbb P}_m(\bar Y(\infty)>u)={\mathbb P}_m(\tau(u)<\infty) \end{align*}
in the regime that $u$ grows large, with ${\mathbb P}_m(\cdot)$ defined in the evident manner,
and with \[\tau(u) :=\inf\{t>0: Y(t)>u\}\]  denoting the {\it exceedance time} of level $u$ (which can be defective). 
Here we restrict ourselves to the setting that the spectrally-positive L\'evy processes $Z_n(\cdot)$ correspond to deterministic drifts, i.e., $Y(\cdot)\in {\mathscr L} [m, {\bs \lambda}, {\bs {\mathscr B}}(\cdot),{\bs r}]$. 
It is remarked, though, that one could in principle also consider the broader class $Y(\cdot)\in {\mathscr L} [m, {\bs \lambda}, {\bs {\mathscr B}}(\cdot),{\bs \varphi}(\cdot)]$, but that would require distinguishing a large number of cases; the precise shapes of each of the Laplace exponents $\varphi_0(\cdot),\ldots, \varphi_m(\cdot)$ and the LST $\bs{\mathscr B}(\cdot)$ all have impact on the behavior of $p_m(u,\beta)$ and $p_m(u)$ for $u$ large.

In our analysis of the tail asymptotics we specifically consider two important classes of claim-size distributions: phase-type claims and regularly-varying claims. Phase-type distributions owe their relevance to the fact that they can approximate any distribution on the positive half-line arbitrarily closely \cite[Thm.\ III.4.2]{ASM2}. Regularly-varying distributions provide a natural framework for modeling heavy-tailed phenomena~\cite{BGT}.

\begin{remark}\label{remark: applic}\em     
    In terms of the application we have in mind, it may feel rather unnatural that we work with claim arrival rates and claim-size distributions that depend on  the number of major claims that were submitted so far (where it is noted that this does {\it not} mean that each major client is allowed to have their own claim-size distribution). In practice, the most relevance instance of our model is probably the one in which $\lambda^\circ_n=\lambda n$ and $B_1,\ldots,B_m$ being distributed as the {\it same} random variable $B$; this is the situation of $m$ independent and statistically identical claim sizes, each of them arriving after an exponentially distributed time with mean $\lambda^{-1}.$
    The proof of our main theorem on $\pi_n(\alpha,\beta)$, Theorem~\ref{T2}, however reveals that working with the more general version has a significant added value. 
    
    As is clear from our model description, we also allow the claim process corresponding to the small clients to depend on the number of major claims submitted so far; the motivation of including this feature is that large claims being submitted may affect the behavior of the small clients. 

    To visually illustrate our model, we have plotted  in Figure \ref{fig: plot single path} an instance of a single path from the model described in this remark. Here, the general arrival rate is $\lambda=1$, the L\'evý processes $Z_n(\cdot)$ are Brownian motions with positive drifts $r_n=n$ and constant variance parameters $\sigma_n^2 = 1$, and the claims are identically exponentially distributed with parameter $\mu = \frac{1}{10}$. \hfill$\Diamond$
\end{remark}

\begin{figure}[ht]
    \centering
    \includegraphics{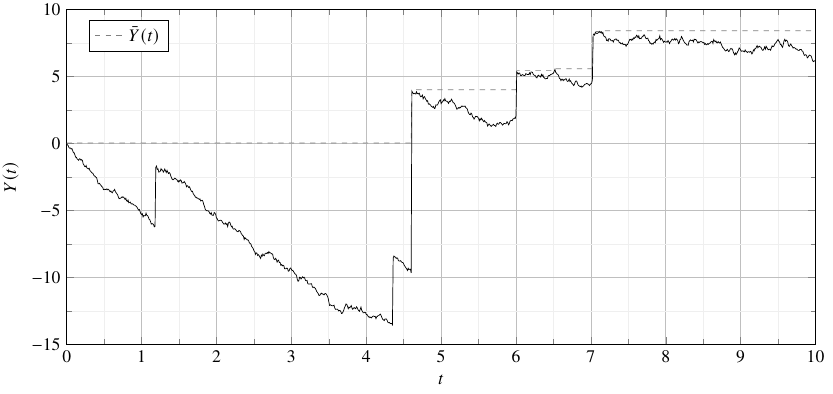}
    \caption{\label{fig: plot single path}An example of an instance of a single path in the model with Brownian motions with positive drifts, for $m=5$.}
\end{figure}

\section*{III. Laplace-Stieltjes transform of running maximum}
The main contribution of this section concerns a recursive procedure via which one can determine the LST $\pi_n(\alpha,\beta)$, with $n\in\{0,\ldots,m\}$, for the net cumulative claim processes $Y(\cdot)\in {\mathscr L} [m, {\bs \lambda}, {\bs {\mathscr B}}(\cdot),{\bs \varphi}(\cdot)]$. 
We prove this result by first establishing,  in Section III.1, a result for the LST of a relatively elementary random variable $Y^\circ$. It turns out that $\bar Y(T_\beta)$ can be written in terms of this random variable $Y^\circ$ for $Y(\cdot)\in {\mathscr L} [m, {\bs \lambda}, {\bs {\mathscr B}}(\cdot),{\bs r}]$, as detailed in Section III.2.
Remarkably, by another specific choice of the model primitives, in combination with the use of the Wiener-Hopf decomposition \cite[Chapter VI]{KYP}, also the more general class $Y(\cdot)\in {\mathscr L}[m, {\bs \lambda}, {\bs {\mathscr B}}(\cdot),{\bs \varphi}(\cdot)]$ can be treated by using the result for $Y^\circ$, as pointed out in Section~III.3. We recall that in Sections~IV and V, where we study the tail asymptotics of $p_m(u,\beta)$ and $p_m(u)$ in the regime that $u$ grows large, we exclusively work with the more narrow class $Y(\cdot)\in {\mathscr L} [m, {\bs \lambda}, {\bs {\mathscr B}}(\cdot),{\bs r}]$.

\subsubsection*{III.1.~~Generic result} In this subsection we define the generic random variable $Y^\circ$ which will, later in this section, facilitate the analysis of $\pi_n(\alpha,\beta)$. This $Y^\circ$ is defined as the maximum of a random walk with a stochastic number of terms, in which each term is the difference between a generally distributed quantity and an exponentially distributed quantity. 

We proceed by providing a formal definition of $Y^\circ.$ To this end, we consider a sequence of independent random variables $C_1,\ldots,C_m$, a probability vector ${\bs p}\equiv (p_0,\ldots,p_m)\in{\mathbb R}^{m+1}$  (i.e., a vector with entries being non-negative and adding up to one), and  a component-wise positive vector ${\bs \nu}\equiv (\nu_1,\ldots,\nu_m)\in{\mathbb R}^m$. Denote by ${\mathscr C}_n(\cdot)$  the LST of $C_n$, and 
$ {\bs{\mathscr C}}(\cdot)\equiv({\mathscr C}_1(\cdot),\ldots,  {\mathscr C}_m(\cdot))$.
Let $T_{\nu_1},\ldots, T_{\nu_m}$ be independent exponentially distributed random variables with means $\nu_1^{-1},\ldots,\nu_m^{-1}$, respectively, independent of $  {\bs C}\equiv (C_1,\ldots,C_m).$ Let the random variable $M$, attaining values in $\{0,1,\ldots,m\}$, be sampled independently of everything else, with ${\mathbb P}(M=j)=p_j$.
Then define
\[Y^\circ \equiv Y^\circ(m, {\bs \nu},  {\bs{\mathscr C}}(\cdot), {\bs p}) = \max_{j=0,\ldots,M} \sum_{i=1}^j \big(C_i - T_{\nu_i}\big)\, \]
where we follow the convention that an empty sum is defined as zero. 
The main goal of this subsection is to devise a procedure to compute the LST of $Y^\circ$:
\[\pi^\circ(\alpha) := {\mathbb E}\,e^{-\alpha Y^\circ}.\] We do so by first analyzing the Laplace transform pertaining to the complementary distribution function ${\mathbb P}(Y^\circ>u)$, defined as 
\[\varrho^\circ(\alpha) :=\int_0^\infty e^{-\alpha u}\,{\mathbb P}(Y^\circ>u)\,{\rm d}u. \]
In our derivations, a number of auxiliary objects play a key role. We define, with $q_{n}:= 1-p_0-\cdots-p_{n-1}$ (where $q_0:=1$), for $n\in\{1,\ldots,m\}$,
\begin{align*}
    {\bs \nu}^{(n)}&:= (\nu_{m-n+1},\ldots,\nu_m);\\
   {\bs {\mathscr C}}^{(n)}(\cdot)&:= ({\mathscr C}_{m-n+1}(\cdot),\ldots,{\mathscr C}_m(\cdot));\\
    {\bs p}^{(n)}&:= \left(\frac{p_{m-n}}{q_{m-n}},\ldots, \frac{p_m}{q_{m-n}}\right);
\end{align*}
these are the counterparts of the objects ${\bs \nu}$, ${\bs{\mathscr C}}(\cdot)$, and ${\bs p}$ when $n$ (rather than $m$) clients have not yet submitted their claim; in this context, notice that the first entry of ${\bs p}^{(n)}$ equals
\[\frac{p_{m-n}}{q_{m-n}}={\mathbb P}(M=m-n\,|\,M\geqslant m-n).\]
Observe that ${\bs \nu}^{(n)}$ and ${\bs {\mathscr C}}^{(n)}(\cdot)$ are $n$-dimensional, and that ${\bs p}^{(n)}$ is $(n+1)$-dimensional, making the random variable $Y^\circ_n:=Y^\circ(n,{\bs \nu}^{(n)}, {\bs {\mathscr C}}^{(n)}(\cdot), {\bs p}^{(n)})$ properly defined. 
Then define its LST via
\[\pi^\circ_n (\alpha):={\mathbb E}\,e^{-\alpha Y^\circ_n},\]
so that $\pi^\circ (\alpha)=\pi^\circ_m (\alpha)$, and analogously
\[\varrho_n^\circ(\alpha) :=\int_0^\infty e^{-\alpha u}\,{\mathbb P}(Y^\circ_n>u)\,{\rm d}u,\] 
so that $\varrho^\circ (\alpha)=\varrho^\circ_m (\alpha)$. Our objective is to express $\pi^\circ_n (\alpha)$ in terms of $\pi^\circ_{n-1} (\alpha)$, which we do by expressing $\varrho^\circ_n (\alpha)$ in terms of $\varrho^\circ_{n-1} (\alpha)$. In this context, the following (well-known) result is particularly useful; see e.g.\ \cite[Eqn.\ (1.1)]{MB}.
\begin{lemma}\label{lemma: relation pi_n rho_n}
    For any $\alpha\geqslant0$ and for any non-negative random variable $Y$, 
    \begin{equation}\label{eq: relation pi_n rho_n}
        {\mathbb E} \,e^{-\alpha Y}= 1 - \alpha\int_0^\infty e^{-\alpha u} \,{\mathbb P}(Y > u)\,{\rm d} u.
    \end{equation}
\end{lemma}

The idea is that we evaluate $\varrho^\circ_n (\alpha)$ by conditioning on whether or not $M$ is equal to $0$. By decomposing into the events that the level $u$ is reached due to $C_1$ or not, we have, 
    \begin{align}\label{eq: decomp rho_nM}
        \varrho^\circ_m(\alpha) =\:& (1-p_{0}) \int_0^\infty e^{-\alpha u} \int_0^\infty \nu_1 e^{-\nu_1 t} \,\mathbb{P}\left(C_1 > u+t\right) \,\mathrm{d}t\,\mathrm{d}u \:+\\
        \notag& (1-p_{0}) \int_0^\infty e^{-\alpha u} \int_0^\infty \nu_1 e^{-\nu_1 t} \int_0^{u+t}\mathbb{P}(C_1\in \mathrm{d}s)\, \mathbb{P}\left(Y^\circ_{m-1}> u+t -  s\right)\,\mathrm{d}t\,\mathrm{d}u.
    \end{align}
    The first term on the right hand side of \eqref{eq: decomp rho_nM} can be calculated directly. Rewriting the probability within the integral gives, changing the order of integration and substituting $w = u+t$,
    \begin{align*}
        (1-p_0) &\int_0^\infty e^{-\alpha u} \int_0^\infty \nu_1 e^{-\nu_1 t} \,\left(1-\mathbb{P}\left(C_1 < u+t\right)\right) \,\mathrm{d}t\,\mathrm{d}u \\
        &=(1-p_0)\left(\frac{1}{\alpha} + \frac{\nu_1}{\alpha-\nu_1}\left(\frac{{\mathscr C}_1(\nu_1)}{\nu_1}-\frac{{\mathscr C}_1(\alpha)}{\alpha}\right)\right)\\
        &=\left(1-p^{(m)}_0\right)\left(\frac{1}{\alpha} + \frac{\nu^{(m)}_1}{\alpha-\nu_1^{(m)}}\left(\frac{{\mathscr C}^{(m)}_1(\nu_1^{(m)})}{\nu_1^{(m)}}-\frac{{\mathscr C}_1^{(m)}(\alpha)}{\alpha}\right)\right).
    \end{align*}
    For the second term on the right side of \eqref{eq: decomp rho_nM}, once more interchanging the order of integration and substituting $w = u+t$, we obtain along similar lines that
    \begin{equation}\label{eq: 2nd rhs rho_n}
        \left(1-p^{(m)}_0\right) \frac{\nu_1^{(m)}}{\nu_1^{(m)} - \alpha}\left( {\mathscr C}_1^{(m)}(\alpha){\varrho}^\circ_{m-1}(\alpha) -  {\mathscr C}_1^{(m)}(\nu_1^{(m)}){\varrho}^\circ_{m-1}(\nu_1^{(m)})\right).
    \end{equation}
The above reasoning shows how to express $\varrho^\circ_m(\alpha)$ in terms of $\varrho^\circ_{m-1}(\alpha)$. Proceeding along in the same manner provides us with the following lemma. 
\begin{lemma}
    For any $\alpha\geqslant 0$ and $n\in\{1,\ldots,m\}$, 
    \begin{align*}
        \varrho^\circ_n(\alpha)=\:& \left(1-p^{(n)}_0\right)\left(\frac{1}{\alpha} + \frac{\nu^{(n)}_1}{\alpha-\nu_1^{(n)}}\left(\frac{{\mathscr C}^{(n)}_1(\nu_1^{(n)})}{\nu_1^{(n)}}-\frac{{\mathscr C}_1^{(n)}(\alpha)}{\alpha}\right)\right)\:+\\
        \:& 
        \left(1-p^{(n)}_0\right) \frac{\nu_1^{(n)}}{\nu_1^{(n)}-\alpha}\left( {\mathscr C}_1^{(n)}(\alpha){\varrho}^\circ_{n-1}(\alpha) -  {\mathscr C}_1^{(n)}(\nu_1^{(n)}){\varrho}^\circ_{n-1}(\nu_1^{(n)})\right),
    \end{align*}
    with ${\varrho}^\circ_0(\alpha) = 0$.
\end{lemma}
We proceed by providing the analogous recursion for $\pi^\circ_n(\alpha)$. Rewriting the term $\varrho_n^\circ(\alpha)$ in terms of $\pi_n^\circ(\alpha)$ using the relation in Lemma \ref{lemma: relation pi_n rho_n} yields the following proposition.

\begin{proposition}\label{P1}
    For any $\alpha\geqslant0$ and $n\in\{1,\dots,m\}$,
    \begin{equation}\label{P1_eq}
        \pi_n^\circ(\alpha) = p_0^{(n)} + \left(1-p_0^{(n)}\right) \frac{\nu_1^{(n)}}{\nu_1^{(n)}-\alpha} \left({\mathscr C}_1^{(n)}(\alpha) \pi_{n-1}^\circ(\alpha) - \frac{\alpha}{\nu_1^{(n)}}{\mathscr C}_1^{(n)}(\nu_1^{(n)})\pi_{n-1}^\circ(\nu_1^{(n)}) \right),
    \end{equation}
    with $\pi_0^\circ(\alpha) = 1$. 
\end{proposition}

The recursion of Proposition \ref{P1} can be solved in a straightforward manner, as we demonstrate now; in particular there are no serious complications related to determining the unknown constants appearing in \eqref{P1_eq}. 
For $n=1$ we obtain the explicit expression 
\[
\pi_1^\circ(\alpha) = p_0^{(1)} + \left(1-p_0^{(1)}\right) {\mathscr D}(\alpha),\:\:\mbox{with}\:\:\:{\mathscr D}(\alpha):=\frac{\nu_1^{(1)}}{\nu_1^{(1)}-\alpha} \left({\mathscr C}_1^{(1)}(\alpha)  - \frac{\alpha}{\nu_1^{(1)}}{\mathscr C}_1^{(1)}(\nu_1^{(1)}) \right).
\]
Iterating once more,
\begin{align*}
    \pi_2^\circ(\alpha) = p_0^{(2)} + \left(1-p_0^{(2)}\right) \frac{\nu_1^{(2)}}{\nu_1^{(2)}-\alpha}\Bigg(&  {\mathscr C}_1^{(2)}(\alpha)\left(  p_0^{(1)} + \left(1-p_0^{(1)}\right) {\mathscr D}(\alpha)\right)\:- \\
    &\frac{\alpha}{\nu_1^{(2)}}{\mathscr C}_2^{(2)}(\nu_1^{(2)})
    \left(  p_0^{(1)} + \left(1-p_0^{(1)}\right) {\mathscr D}(\nu_1^{(2)})\right) \Bigg).
\end{align*}
If $\nu_1^{(1)}\not= \nu_1^{(2)}$, then we evidently have
\[{\mathscr D}(\nu_1^{(2)}):=\frac{\nu_1^{(1)}}{\nu_1^{(1)}-\nu_1^{(2)}} \left({\mathscr C}_1^{(1)}(\nu_1^{(2)})  - \frac{\nu_1^{(2)}}{\nu_1^{(1)}}{\mathscr C}_1^{(1)}(\nu_1^{(1)}) \right)=\frac{\nu_1^{(1)} {\mathscr C}_1^{(1)}(\nu_1^{(2)}) - \nu_1^{(2)} {\mathscr C}_1^{(1)}(\nu_1^{(1)})}{\nu_1^{(1)}-\nu_1^{(2)}},\]
whereas in the case that $\nu_1^{(1)}= \nu_1^{(2)}$ this quantity can be computed by an elementary application of L'H\^opital's rule. Proceeding along these lines, all LST\,s $\pi_0^\circ(\alpha),\ldots,\pi_m^\circ(\alpha)$ can be identified.

\begin{corollary}
    For any $n\in\{1,\dots,m\}$,
    \[ {\mathbb P}({Y}_n^\circ=0)= \pi_n^\circ(\infty) = p_0^{(n)} + \left(1-p_0^{(n)}\right) \left({\mathscr C}_1^{(n)}(\nu_1^{(n)})\pi_{n-1}^\circ(\nu_1^{(n)}) \right).\]
\end{corollary}

\subsubsection*{III.2.~~Small clients corresponding to deterministic drifts}
After having developed the recursive algorithm that produces the LST of $Y^\circ$, we now go back to the model introduced in Section~II. In this subsection we discuss the special instance in which the small clients are represented by deterministic drifts, i.e., $Y(\cdot)\in {\mathscr L} [m, \bs\lambda, \bs{\mathscr B}(\cdot),{\bs r}]$, where we can effectively directly apply Proposition \ref{P1}.

For ease we concentrate on the case that all entries of ${\bs r}$ are strictly positive; we later discuss the case in which this does not hold for some entries. 
We thus have that  $Z_n(t) = -r_n t$ for $r_n>0$, $n\in\{0,\ldots,m\}$, and hence, trivially,
\begin{equation*}
    \varphi_n(\alpha) = \log \mathbb{E}\,e^{-\alpha Z_n(1)} = \log \mathbb{E}\,e^{\alpha r_n} = \alpha r_n.
\end{equation*}
The idea is to represent $\bar Y(T_\beta)$ in terms of the random variable $Y^\circ$, with a specific choice of the parameters underlying $Y^\circ$. To this end we choose, for $n\in\{1,\ldots,m\}$,
\[{\mathscr C}_1^{(n)}(\cdot)={\mathscr B}_{m-n+1}(\cdot),\:\:\:\nu_1^{(n)}=\frac{\lambda_{n}}{r_{n}},\:\:\:p_0^{(n)}= \frac{\beta}{\lambda_{n}}.\]
Application of Proposition \ref{P1} now provides us with a recursion for $\pi_n(\alpha,\beta).$

\begin{theorem}\label{T1}
    Suppose $Y(\cdot)\in {\mathscr L} [m, \lambda, {\bs {\mathscr B}}(\cdot),{\bs r}]$. For any $\alpha\geqslant0$, $\beta\geqslant0$, and $n\in\{1,\dots,m\}$,
    \begin{equation}\label{T1_eq}
        \pi_n(\alpha,\beta) =  \frac{\beta}{\lambda_{n}} + \frac{\lambda_n^\circ}{\lambda_n-\alpha r_n} \left({\mathscr B}_{m-n+1}(\alpha)\, \pi_{n-1}(\alpha,\beta) - \frac{\alpha r_n}{\lambda_n}{\mathscr B}_{m-n+1}\Big(\frac{\lambda_{n}}{r_{n}}\Big)\,\pi_{n-1}\Big(\frac{\lambda_{n}}{r_{n}},\beta\Big) \right),
    \end{equation}
    with $\pi_0(\alpha,\beta) = 1$. 
\end{theorem}

The LST $\pi_m(\alpha, \beta)$ in principle uniquely characterizes the probability density of the running maximum as functions of the initial reserve level $u$ and the time $t$. However, the transforms derived from Theorem \ref{T1} often lead to complex expressions involving multiple variables, making symbolic double inversion a challenging task. Therefore, it is generally more practical to employ {\it numerical} Laplace inversion techniques.

We proceed by discussing how Theorem \ref{T1} can be used to evaluate various performance metrics pertaining to the running maximum process $\bar Y(\cdot)$. In our experiment, we use the parameters of the model that was described in Remark \ref{remark: applic}. We take an arrival rate $\lambda = \frac{1}{4}$ for each individual major claim, so that $\lambda^\circ_n = \lambda n = \frac{n}{4}$. The claim sizes are independent and exponentially distributed with parameter $\mu = \frac{1}{4}$, and for the drifts we choose $r_n = n$. Taking the first and second derivatives of $\pi_m(\alpha,\beta)$ with respect to $\alpha$, and letting $\alpha=0$, we readily obtain  recursions from which the mean $\mathbb{E}_m\,\bar{Y}(T_\beta)$ and variance ${\rm Var}_m\,\bar{Y}(T_\beta)$ can be evaluated; recall that, after dividing by $\beta$, these are to be interpreted as (single) Laplace transforms. By applying Stehfest's numerical inversion algorithm (see \cite{STEH}, or, for the specific algorithm used here, \cite[Proposition 8.2]{AW2}), we obtain the mean $\mathbb{E}_m\,\bar{Y}(\cdot)$ and variance ${\rm Var}_m\,\bar{Y}(\cdot)$ of the running maximum as functions of time. Figure \ref{fig: base model} presents the results of the numerical inversion for different values of $m$.

\begin{figure}[ht]
  \centering
\resizebox{10cm}{4.5cm}{\includegraphics{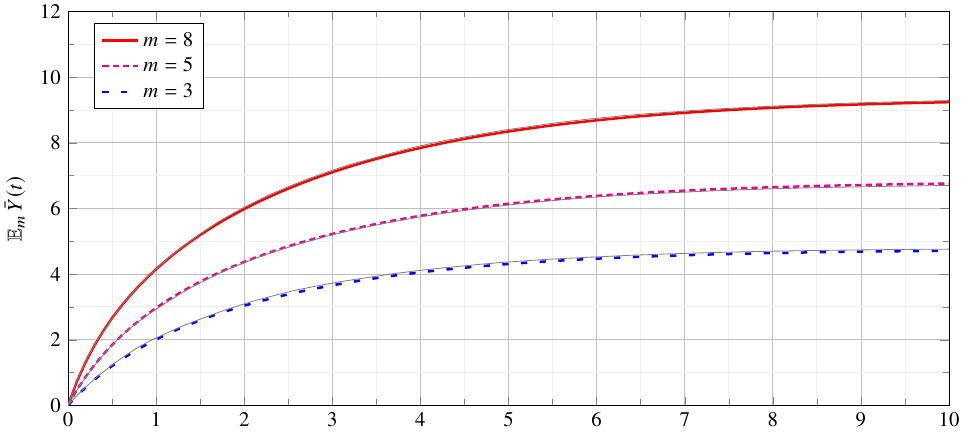}}

\resizebox{10cm}{4.5cm}{\includegraphics{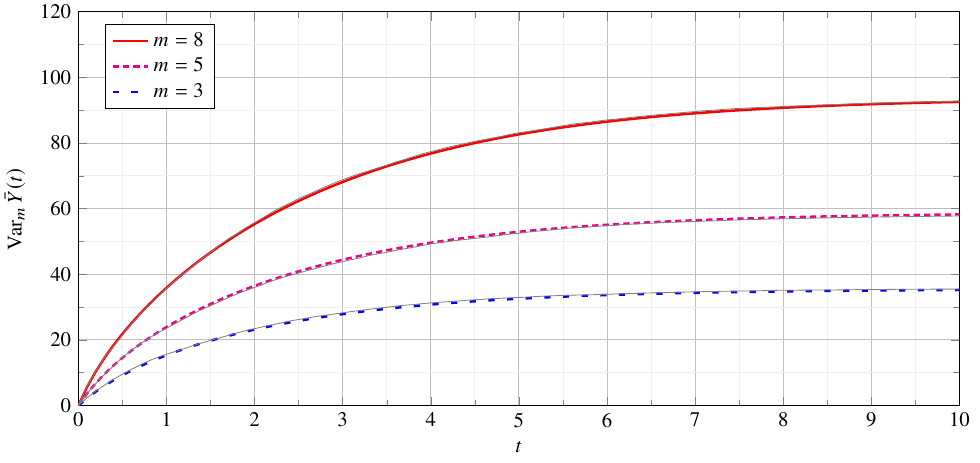}}
    \caption{\label{fig: base model}Mean (top) and variance (bottom) of the running maximum, as functions of time, in the model with only positive drifts, for different values of $m$. The gray lines denote Monte Carlo based estimated values and are plotted for comparison.}
\end{figure}

\subsubsection*{III.3.~~Small clients corresponding to spectrally-positive L\'evy processes}
We expand on the model in the previous subsection by going back to the complete model introduced in Section II, for which the small clients are modeled by general spectrally positive Lévy processes. We let the Lévy process $Z_n(t)$ be characterized by its Laplace exponent $\varphi_n(\cdot)$ and its inverse $\psi_n(\cdot)$ as stated in Section II, so that $Y(\cdot)\in {\mathscr L} [m, \bs\lambda, \bs{\mathscr B}(\cdot),{\bs \phi}(\cdot)]$. 
We let the random variable $M$ denote the number of large claims before the killing time $T_\beta$, which is characterized, for $n=0,\dots,m$, by 
\begin{align*}
    {\mathbb P}(M=n) = \left(\prod_{i=1}^{n}\frac{\lambda^\circ_{m-i+1}}{\lambda_{m-i+1}}\right)\frac{\beta}{\lambda_{m-n}}=\left(\prod_{i=m-n+1}^{m}\frac{\lambda^\circ_{i}}{\lambda_{i}}\right)\frac{\beta}{\lambda_{m-n}},
\end{align*}
with the empty product being defined as one and $\lambda^\circ_0 := 0$.
\begin{lemma}\label{L3}
    Suppose $Y(\cdot)\in {\mathscr L} [m, \bs\lambda, \bs{\mathscr B}(\cdot),{\bs \phi}(\cdot)]$. For any $\beta>0$
    \begin{equation}\label{eq: decomp} \bar Y(T_\beta)\stackrel{{\rm d}}{=} \bar Z_m(T_{\lambda_m}) + \max_{j=0,\ldots,M} \sum_{i=1}^j\big(B_i+\bar Z_{m-i}(T_{\lambda_{m-i}}) - T_{\psi_{m-i+1}(\lambda_{m-i+1})}\big),\end{equation}
with all random variables appearing in the right hand side being independent. \end{lemma}
\begin{proof}
    Observe that $\bar Y(T_\beta)$ is attained at a time that one of the L\'evy processes $Z_n(\cdot)$ attains its maximum value in the interval in which it is active, i.e., the time interval $[A_{(n)},A_{(n+1)})$. More concretely, 
    \begin{equation}\label{eq: decomp1}\bar Y(T_\beta) \stackrel{\rm d}{=} \max\left\{\bar Z_m(T_{\lambda_m}),\max_{j=0,\ldots,M} \left(\sum_{i=1}^j \big(B_{i}+Z_{m-i+1}(T_{\lambda_{m-i+1}})\big)+\bar Z_{m-j}(T_{\lambda_{m-j}})\right)\right\}.\end{equation}
    The next step is to recall that the Wiener-Hopf decomposition entails that, for a spectrally positive L\'evy process $Z_n(\cdot)$, the random variable $Z_n(T_{\lambda})$ has the same distribution as the running maximum  $\bar Z_n(T_{\lambda})$ minus an exponentially distributed random variable with mean $1/\psi_n(\lambda_n)$ \cite[\S 6.5.2]{KYP}. This means that we can write
    \[Z_{m-i}(T_{\lambda_{m-i}}) \stackrel{\rm d}{=} \bar Z_{m-i}(T_{\lambda_{m-i}}) - T_{\psi_{m-i}(\lambda_{m-i})},\]
    with the two terms on the right hand side being independent. By inserting this into the right hand side of \eqref{eq: decomp1}, and reordering the terms, we obtain the expression given in \eqref{eq: decomp}.
\end{proof}
Define, for $n\in\{0,\ldots,m\}$,
\[{\mathscr Z}_n(\alpha, \lambda):= {\mathbb E}\, e^{-\alpha \bar Z_n(T_\lambda)}= \frac{\psi_n(\lambda) - \alpha}{\lambda - \varphi_n(\alpha)}\frac{\lambda}{\psi_n(\lambda)},\]
where we once more appealed to \cite[\S 6.5.2]{KYP}. Application of Proposition \ref{P1} then provides us with a recursion by which $\pi_m(\alpha,\beta)$ can be determined, as given in the following theorem. 
\begin{theorem}\label{T2}
    Suppose $Y(\cdot)\in {\mathscr L} [m, {\bs \lambda}, {\bs {\mathscr B}}(\cdot),{\bs \varphi}(\cdot)]$. Then we have, for any $\alpha\geqslant0$, $\beta>0$, and $n\in\{ 0,\dots,m\}$,
     \begin{equation} \label{eq: pim}
        \pi_n(\alpha, \beta) = \mathscr{Z}_n(\alpha,\lambda_n)\,\pi^\circ_n(\alpha,\beta).
    \end{equation}
    where
    \begin{align}\notag
        \pi^\circ_n(\alpha,\beta) = \: \frac{\beta}{\lambda_{n}} + \frac{\lambda_n^\circ}{\lambda_n}&\frac{\psi_n(\lambda_n)}{\psi_n(\lambda_n)-\alpha} \Bigg({\mathscr B}_{m-n+1}(\alpha)\,\mathscr{Z}_{n-1}(\alpha,\lambda_{n-1})\, \pi^\circ_{n-1}(\alpha,\beta) \:-\\
        &\frac{\alpha }{\psi_n(\lambda_n)}{\mathscr B}_{m-n+1}(\psi_n(\lambda_n))\,\mathscr{Z}_{n-1}(\psi_n(\lambda_n),\lambda_{n-1})\,\pi^\circ_{n-1}(\psi_n(\lambda_{n}),\beta) \Bigg)\label{T2_eq}
    \end{align}
    for $n\in\{1,\dots,m\}$ and $\pi_0^\circ(\alpha,\beta)=1.$
\end{theorem}

\begin{proof}
    Let $\pi_n^\circ(\alpha,\beta)$ be the LST of 
    \[\max_{j=0,\ldots,M} \sum_{i=1}^j\big(B_i+\bar Z_{n-i}(T_{\lambda_{n-i}}) - T_{\psi_{n-i+1}(\lambda_{n-i+1})}\big),\]
    for $n\in\{0,\ldots,m\}.$    Now applying Lemma \ref{L3}, we immediately obtain the identity \eqref{eq: pim}. It is left to prove that $\pi^\circ_m(\alpha,\beta)$ can be evaluated using the recursion \eqref{T2_eq}. Observe that this recursion is now a direct consequence of Proposition \ref{P1}, with
    \[{\mathscr C}_1^{(n)}(\cdot)={\mathscr B}_{m-n+1}(\cdot)\,{\mathscr Z}_{n-1}(\cdot, \lambda_{n-1}),\:\:\:\nu_1^{(n)}=\psi_n({\lambda_{n}}),\:\:\:p_0^{(n)}= \frac{\beta}{\lambda_{n}},\]
     for $n\in\{1,\ldots,m\}$. 
 \end{proof}

We have plotted in Figure \ref{fig: model with BM} the mean and variance of the running maximum as obtained through numerical inversion. We used the same model parameters, i.e., $\lambda^\circ_n = \frac{n}{4}$, $r_n=n$ and exponentially distributed claims with parameter $\mu=\frac{1}{4}$, but now a Brownian motion with variance parameter $\sigma^2=1$ has been added to each of the drift processes. For the numerical inversion  we again relied on Stehfest's inversion algorithm. The dotted lines in the figures show the results from the model {\it without} Brownian motions, which are the same as in Figure~\ref{fig: base model}. The graphs quantify by how much the mean $\mathbb{E}_m\,\bar{Y}(\cdot)$ and variance ${\rm Var}_m\,\bar{Y}(\cdot)$ increase due to the added Brownian term.

\begin{figure}[ht]
    \centering
    \resizebox{10cm}{4.5cm}{\includegraphics{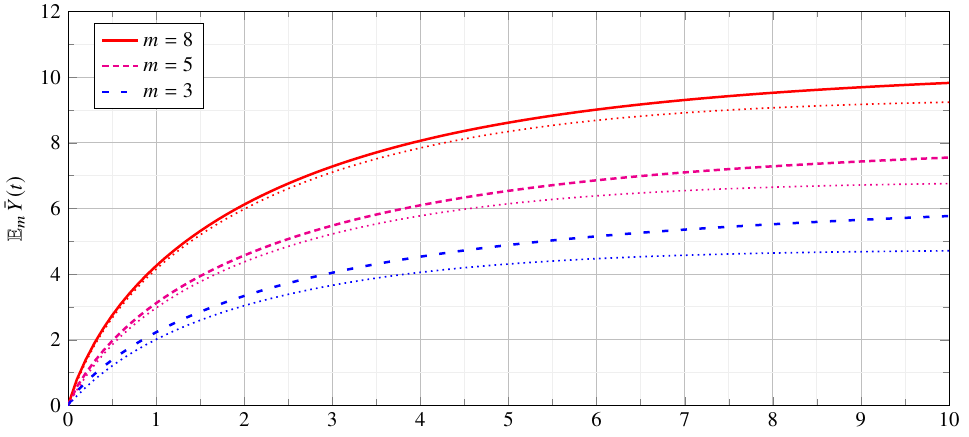}}

    \resizebox{10cm}{4.5cm}{\includegraphics{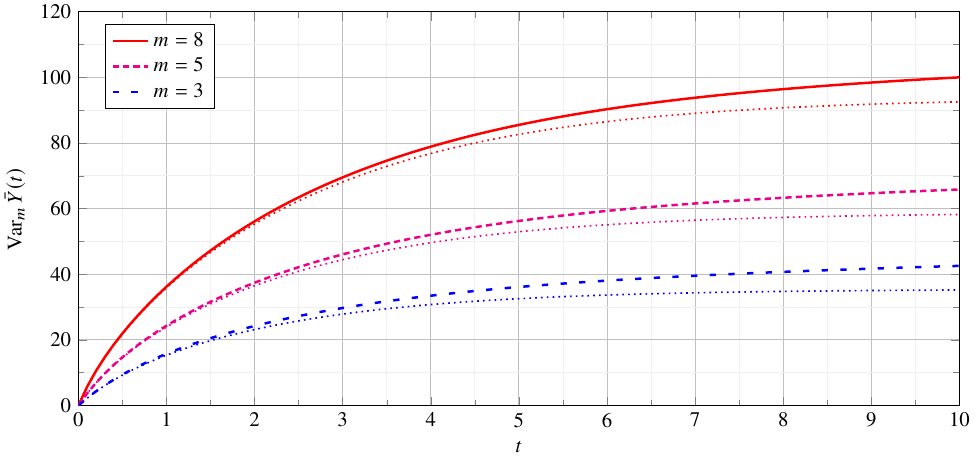}}
    
    \caption{\label{fig: model with BM}Mean (top) and variance (bottom) of the running maximum, as functions of time, in the model with Brownian motions with positive drifts, for different values of $m$. The dotted lines denote the model without the added Brownian motions.}
\end{figure}

\subsubsection*{III.4.~~Extension to the case with subordinator states}
Suppose $n\in\{0,\ldots,m\}$ is a `subordinator state', i.e., the L\'evy process $Z_n(\cdot)$ is almost surely increasing. It is noted that this case has to be dealt with separately, as the right inverse $\psi_n(\cdot)$ is ill-defined. However, evidently we have $\bar Z_n(t) =Z_n(t)$ for all $t\geqslant 0$, and therefore
\[ {\mathbb E}\, e^{-\alpha \bar Z_n(T_\lambda)}= \frac{\lambda}{\lambda - \varphi_n(\alpha)}.\]
It is now directly seen that one has to replace the recursion step \eqref{P1_eq} by
\begin{align*} \pi^\circ_n(\alpha,\beta) &=\frac{\lambda_n}{\lambda_n - \varphi_n(\alpha)} \left(\frac{\beta}{\lambda_n}+\frac{\lambda^\circ_n}{\lambda_n} \,{\mathscr B}_{m-n+1}(\alpha)\,\pi_{n-1}^\circ(\alpha,\beta)\right)=\frac{\beta+\lambda^\circ_n \,{\mathscr B}_{m-n+1}(\alpha)\,\pi_{n-1}^\circ(\alpha,\beta)}{\lambda_n - \varphi_n(\alpha)} .
\end{align*}

\section*{IV. Phase-type claims}
In this section and the next one, we consider the tail behavior of the ruin probabilities $p_m(u,\beta)$ and $p_m(u)$; the present section specifically focuses on the case that the claim-size distribution is of phase-type~\cite{BFN}. The phase-type distributions owe their popularity to the fact that they constitute a class of distributions by which one can approximate the distribution function of any random variable on the positive half-line arbitrarily closely \cite[Thm.\ III.4.2]{ASM2}.  

We proceed by introducing the notation that we use throughout this section, essentially following the setup of \cite[\S III.4]{ASM2}.
A phase-type distribution is characterized by a dimension $d\in{\mathbb N}$, a probability that the variable is zero, $0\leqslant\delta_{d+1}\leqslant 1$, a $d$-dimensional vector ${\bs\delta}$ with non-negative entries summing to $1-\delta_{d+1}$, and a $d\times d$ matrix ${\bs S}.$ The matrix ${\bs S}$, known as the {\it phase generator matrix}, is a defective rate matrix (in that at least one of the row sums is strictly negative), and the exit vector ${\bs s}$ is defined as $-{\bs S}{\bs 1}_d,$ with ${\bs 1}_d$ a $d$-dimensional all-ones column vector. The vector ${\bs\delta}$ is known as the {\it initial distribution vector}. With  the transition rate matrix
\[{\bs Q}:=\left(\begin{array}{cc}
{\bs S}&{\bs s}\\
{\bs 0}_d^\top&0
\end{array}\right),\]
and ${\bs 0}_d$ the $d$-dimensional all-zeroes column vector, a phase-type random variable is defined as the time it takes for a continuous-time Markov chain $(J_t)_{t\geqslant 0}$ with initial distribution ${\bs \delta}$ and transition rate matrix ${\bs Q}$ to reach the absorbing state $d+1$:
\[U:=\inf\{t\geqslant 0: J_t=d+1\}.\]
In the sequel we write $U\sim {\mathbb P}{\rm h}_d({\bs\delta},{\bs S})$.
Observe that our definition slightly differs from the one given in \cite[\S III.4]{ASM2}, in that we allow the random variable to attain the value $0$ with positive probability (captured by $\delta_{d+1}$). 
The LST of $U$ follows directly from e.g.\ \cite[Prop.\ III.4.1]{ASM2}, with ${\bs I}_d$ denoting the $d\times d$ identity matrix,
\begin{equation}\label{eq:LST PH}{\mathscr U}(\alpha) :={\mathbb E}\,e^{-\alpha U}= \delta_{d+1} + {\bs\delta}^\top(\alpha {\bs I}_d- {\bs S})^{-1}{\bs s},\end{equation}
for $\alpha\geqslant 0$.

\medskip

In this section we consider the instance that the claim sizes of the large clients, represented by $B_1,\ldots, B_m$, are independent and identically distributed random variables, distributed as the generic random variable $B$ with LST ${\mathscr U}(\cdot)$ as given in \eqref{eq:LST PH}; in other words, ${\bs {\mathscr B}}(\cdot)={\mathscr U}(\cdot)\,{\bs 1}^{\top}.$ We let the small clients' L\'evy processes be given by deterministic drifts, i.e., we are in the situation that $Y(\cdot)\in{\mathscr L}[m,{\bs \lambda},{\mathscr U}(\cdot)\,{\bs 1}^{\top}, {\bs r}]$.

In the next two lemmas we state a couple of convenient properties, that we need in our analysis. The derivations are based on standard matrix calculus, but we include them for completeness. 

\begin{lemma}\label{lemma: phase-type relations}
    For $\alpha,\alpha' \geqslant 0$, we have
   \[\alpha'(\alpha\bs{I} - \bs{S})^{-1} - \alpha(\alpha'\bs{I} - \bs{S})^{-1} = (\alpha'- \alpha)\left((\alpha'\bs{I} - \bs{S})^{-1} + \alpha'(\alpha'\bs{I} - \bs{S})^{-1}(\alpha \bs{I} - \bs{S})^{-1}\right).\]
\end{lemma}
\begin{proof}
    The proof of this identity is rather elementary. Note that
        \begin{align*}
            (\alpha'\bs{I} - \bs{S})\left(\alpha'(\alpha\bs{I} - \bs{S})^{-1} - \alpha(\alpha'\bs{I} - \bs{S})^{-1}\right)(\alpha\bs{I} - \bs{S}) &= \alpha'(\alpha'\bs{I} - \bs{S}) - \alpha(\alpha\bs{I}-\bs{S})\\
            &= (\alpha'- \alpha)((\alpha\bs{I} - \bs{S}) + \alpha'\bs{I}).
        \end{align*}
        The stated follows after pre-multiplying by $(\alpha'\bs{I} - \bs{S})^{-1}$ and post-multiplying by $(\alpha\bs{I} - \bs{S})^{-1}$.
\end{proof}

\begin{lemma}\label{corollary: difference LST's with factors}
    Let $U\sim {\mathbb P}{\rm h}_{d}({\bs\delta},{\bs S})$ with LST $\mathscr{U}(\cdot)$. For any $\alpha,\alpha' \geqslant 0$ we have
    \begin{equation*}
        \frac{\alpha'\mathscr{U}(\alpha) - \alpha\mathscr{U}(\alpha')}{\alpha'-\alpha} = \mathscr{U}(\alpha') + \alpha'\bs\delta^\top\left(\left(\alpha'\bs{I} - \bs{S}\right)^{-1}\left(\alpha \bs{I} - \bs{S}\right)^{-1}\right)\bs{s}.
    \end{equation*}
\end{lemma}
\begin{proof}
    \begin{align*}
        \frac{\alpha'\mathscr{U}(\alpha) - \alpha\mathscr{U}(\alpha')}{\alpha'-\alpha} &= \frac{(\alpha' - \alpha)\delta_{d+1}+\bs\delta^\top\left(\alpha'\left(\alpha \bs{I} - \bs{S}\right)^{-1} - \alpha\left(\alpha' \bs{I} - \bs{S}\right)^{-1}\right)\bs{s} }{\alpha'- \alpha}\\
        &= \delta_{d+1} + \bs\delta^\top\left((\alpha'\bs{I} - \bs{S})^{-1} + \alpha'(\alpha'\bs{I} - \bs{S})^{-1}(\alpha \bs{I} - \bs{S})^{-1}\right){\bs s}\\
        &=\delta_{d+1} + \bs\delta^\top\left(\alpha'\bs{I} - \bs{S}\right)^{-1}\bs{s} + \bs\delta^\top\left(\alpha'\left(\alpha' \bs{I} - \bs{S}\right)^{-1}\left(\alpha \bs{I} - \bs{S}\right)^{-1}\right)\bs{s},
    \end{align*}
    where the second equality follows from Lemma \ref{lemma: phase-type relations}. This proves the statement. 
\end{proof}

\begin{lemma}\label{lemma: sum of two phase types}
    Let $U\sim {\mathbb P}{\rm h}_{d_U}({\bs\delta_U},{\bs S}_U)$ and $V\sim {\mathbb P}{\rm h}_{d_V}({\bs\delta_V},{\bs S}_V)$ be two independent phase-type random variables. Then $W := U+V \sim {\mathbb P}{\rm h}_{d_W}({\bs\delta_W},{\bs S}_W)$, with $d_W := d_U + d_V$, \[\bs\delta_W^\top := (\bs\delta_U^\top, \delta_{U,d_U+1}\bs{\delta}_{V}^\top),\:\:\: \delta_{W,d_W+1} := \delta_{U,d_U+1}\delta_{V,d_V+1},\] and 
    \begin{equation*}
        \bs{S}_W := \begin{pmatrix}
            \bs{S}_U & \bs{s}_U\bs{\delta}_V^\top \\
            \bs{0}_{d_V\times d_U} & \bs{S}_V
        \end{pmatrix} \:\:\text{ with }\: \:\bs{s}_W := \begin{pmatrix}
            \delta_{V,d_V+1}\bs{s}_U\\
            \bs{s}_V
        \end{pmatrix},
    \end{equation*}
    where $\bs{0}_{d_V\times d_U}$ denotes the $d_V\times d_U$ zero matrix.
\end{lemma}
\begin{proof}
    Fixing $\alpha\geqslant0$, our approach is to show that the LST of $W$ in any argument $\alpha$ matches the LST of a random variable with a ${\mathbb P}{\rm h}_{d_W}({\bs\delta_W},{\bs S}_W)$ distribution. Let $\mathscr{U}(\alpha)$ and $\mathscr{V}(\alpha)$ denote the LST\,s of $U$ and $V$ respectively, and let $\mathscr{W}(\alpha)$ denote the LST of $W$. By the independence of $U$ and $V$ we have
    \begin{align}\notag
        \mathscr{W}(\alpha) &= \mathscr{U}(\alpha)\cdot\mathscr{V}(\alpha) =\big(\delta_{U,d_U+1} + {\bs\delta}_U^\top(\alpha {\bs I}_d- {\bs S}_U)^{-1}{\bs s_U}\big)\big(\delta_{V,d_V+1} + {\bs\delta}_V^\top(\alpha {\bs I}_d- {\bs S}_V)^{-1}{\bs s}_V\big)\\
       \notag  &= \delta_{U,d_U+1}\delta_{V,d_V+1} + \delta_{U,d_U+1}\bs\delta_V^\top(\alpha\bs{I} - \bs{S}_V)^{-1}\bs{s}_v + \delta_{V,d_V+1}\bs\delta_U^\top(\alpha\bs{I} - \bs{S}_U)^{-1}\bs{s}_U \\
        &\:\:\:\:\:+ \bs\delta_U^\top(\alpha\bs{I} - \bs{S}_U)^{-1}\bs{s}_U \cdot\bs\delta_V^\top(\alpha\bs{I} - \bs{S}_V)^{-1}\bs{s}_V. \label{eq: idenph}
        \end{align}
        Using the compact notation
    \[{\bs S}_{U,V}:= \begin{pmatrix}
            (\alpha\bs{I} - \bs{S}_U)^{-1} & (\alpha\bs{I} - \bs{S}_U)^{-1}(\bs{s}_U\bs\delta_V^\top) (\alpha\bs{I} - \bs{S}_V)^{-1}\\
            \bs{0}_{d_V\times d_U} &(\alpha\bs{I} - \bs{S}_V)^{-1}
        \end{pmatrix},\] the identity \eqref{eq: idenph} can be rewritten in block matrix form: 
        \begin{align}
        \notag \mathscr{W}(\alpha)&= \delta_{U,d_U+1}\delta_{V,d_V+1} + 
        \begin{pmatrix}
            \bs\delta_U \\
            \delta_{U,d_U+1}\bs{\delta}_V
        \end{pmatrix}^\top 
        {\bs S}_{U,V}
        \cdot
        \begin{pmatrix}
            \delta_{V,d_V+1}\bs{s}_U \\
            \bs{s}_V
        \end{pmatrix} \\
        \notag &= \delta_{U,d_U+1}\delta_{V,d_V+1} + 
        \begin{pmatrix}
            \bs\delta_U \\
            \delta_{U,d_U+1}\bs{\delta}_V
        \end{pmatrix}^\top 
        \cdot
        \begin{pmatrix}
            \alpha\bs{I} - \bs{S}_U & -\bs{s}_U\bs\delta_V^\top \\
            \bs{0}_{d_V\times d_U} & \alpha\bs{I} - \bs{S}_V
        \end{pmatrix}^{-1}
        \cdot
        \begin{pmatrix}
            \delta_{V,d_V+1}\bs{s}_U \\
            \bs{s}_V
        \end{pmatrix} \\
        &= \delta_{W,d_W+1} + \bs\delta_W^\top(\alpha\bs{I} - \bs{S}_W)^{-1}\bs{s}_W.\label{eq: phfin}
    \end{align}
    We recognize in \eqref{eq: phfin} the LST of a phase-type variable with initial distribution vector $\bs\delta_W$ and phase generator matrix $\bs{S}_W$, as defined in the statement of the lemma.
\end{proof}

We proceed by proving that, for $Y(\cdot)\in{\mathscr L}[m,{\bs \lambda},{\mathscr U}(\cdot)\,{\bs 1}^{\top}, {\bs r}]$ and any given $\beta>0$, $\pi_n(\alpha,\beta)$ can be interpreted in terms of the LST of a phase-type random variable. We do so through an inductive proof. We throughout write $\nu_n:=\lambda_n/r_n$. 

\medskip 

{\it --- Step 1.} For $n=1$, by Theorem \ref{T1} and Lemma \ref{corollary: difference LST's with factors}, we find, after a few straightforward algebraic manipulations,
\begin{align*}
    \pi_1(\alpha,\beta) &= \frac{\beta}{\lambda_1} + \frac{\lambda_1^\circ}{\lambda_1} \frac{\nu_1\mathscr{U}(\alpha) - \alpha\mathscr{U}(\nu_1)}{\nu_1 - \alpha} \\
    &= \frac{\beta}{\lambda_1} + \frac{\lambda_1^\circ}{\lambda_1}\left(\mathscr{U}(\nu_1) + \bs\delta^\top\left(\nu_1\left(\nu_1 \bs{I} - \bs{S}\right)^{-1}\left(\alpha \bs{I} - \bs{S}\right)^{-1}\right)\bs{s}\right) = \delta_{1,d+1} +  \bs{\delta}_1^\top\left(\alpha\bs{I} - \bs{S}_1\right)^{-1}\bs{s}
\end{align*}
with, recalling that $\lambda_1=\lambda_1^\circ+\beta$, 
\begin{equation}
    \delta_{1,d+1} := \frac{\beta}{\lambda_1} + \frac{\lambda_1^\circ}{\lambda_1}\mathscr{B}(\nu_1), \hskip0.2cm \bs\delta_1^\top := \frac{\lambda_1^\circ}{\lambda_1}\nu_1 \bs\delta^\top\left(\nu_1\bs{I} - \bs{S}\right)^{-1} \text{ and}\hskip0.2cm \bs{S}_1 := \bs{S}.
\end{equation}
Note that $\bs\delta_1$ is a valid initial distribution vector, as its entries are all positive and
\begin{align*}
    \bs\delta_1^\top\bs{1} + \delta_{1,d+1}&= \frac{\beta}{\lambda_1} + \frac{\lambda_1^\circ}{\lambda_1}\left(\delta_{d+1} + \bs\delta^\top\left(\nu_1 \bs{I} - \bs{S}\right)^{-1}\bs{s}\right) + \frac{\lambda_1^\circ}{\lambda_1}\nu_1\bs\delta^\top\left(\nu_1 \bs{I} - \bs{S}\right)^{-1}\bs{1}\\
    &= \frac{\beta}{\lambda_1} + \delta_{d+1}\frac{\lambda_1^\circ}{\lambda_1} + \frac{\lambda_1^\circ}{\lambda_1}\bs\delta^\top\left(\nu_1 \bs{I} - \bs{S}\right)^{-1}\left( \bs{I}\bs{s} + \nu_1\bs{I}\bs{1}\right) \\
    &= \frac{\beta}{\lambda_1} + \delta_{d+1}\frac{\lambda_1^\circ}{\lambda_1} + \frac{\lambda}{\lambda_1}\bs\delta^\top\left(\nu_1 \bs{I} - \bs{S}\right)^{-1}\left( -\bs{S} + \nu_1\bs{I}\right)\bs{1} \\
    &= \frac{\beta}{\lambda_1} + \delta_{d+1}\frac{\lambda}{\lambda_1} + \frac{\lambda_1^\circ}{\lambda_1}\bs\delta^\top\bs{1} = \frac{\beta}{\lambda_1} + \frac{\lambda_1^\circ}{\lambda_1} = 1.
\end{align*}
Therefore, conditional on $N(0) = 1$, we have that $\bar{Y}(T_\beta)\sim {\mathbb P}{\rm h}_{d}({\bs\delta}_1,{\bs S}_1).$

\medskip

{\it --- Step 2.} In our induction step we assume that, conditional on $N(0)= k$, $\bar{Y}(T_\beta) $ has a ${\mathbb P}{\rm h}_{kd}({\bs\delta}_k,{\bs S}_k)$ distribution for any $k<n$, for some initial distribution vectors $\bs\delta_k$ of length $kd$ and some $kd\times kd$ generator matrices $\bs{S}_k$, with $\bs{s}_k := -\bs{S}_k\bs{1}_{kd}$. Firstly, $\mathscr{U}(\alpha)\,\pi_{n-1}(\alpha,\beta)$ is, by the induction hypothesis, 
 the LST of the sum of two independent phase-type variables. By Lemma \ref{lemma: sum of two phase types} we then find that this sum is also of phase-type, with a distribution that is characterized as ${\mathbb P}{\rm h}_{nd}({\bs\delta}'_n,{\bs S}_n)$, where 
\begin{equation*}
    \bs\delta_n' = \begin{pmatrix}
        \bs\delta_{n-1} \\
        \delta_{n-1,(n-1)d+1}\bs\delta
    \end{pmatrix}
\end{equation*}
is a vector of length $nd$, and 
\begin{equation}\label{eq: recu S_n}
     \bs{S}_n = \begin{pmatrix}
        \bs{S}_{n-1} & \bs{s}_{n-1}\bs{\delta}^\top \\
        \bs{0}_{d\times d} & \bs{S}
    \end{pmatrix} ,
\end{equation}
an $nd\times nd$ matrix; in addition, $\bs{s}_n = -\bs{S}_n\bs{1}_{nd}.$
When conditioning on $N(0)=n$ we obtain, from Theorem \ref{T1} and Lemma \ref{corollary: difference LST's with factors},
\begin{align*}
    \pi_n(\alpha,\beta) &= \frac{\beta}{\lambda_n} + \frac{\lambda_n^\circ}{\lambda_n}\frac{\nu_n\mathscr{U}(\alpha)\pi_{n-1}(\alpha,\beta) - \alpha\mathscr{U}(\nu_n)\pi_{n-1}(\nu_n,\beta)}{\nu_n - \alpha} \\
    &= \frac{\beta}{\lambda_n} + \frac{\lambda_n^\circ}{\lambda_n}\left(\delta_{n, nd+1}' + \bs\delta_n'^\top\left(\nu_n\bs{I} - \bs{S}_n\right)^{-1}\bs{s}_n + \nu_n\bs\delta_n'^\top\left(\nu_n\bs{I} - \bs{S}_n\right)^{-1}\left(\alpha\bs{I} - \bs{S}_n\right)^{-1}\bs{s}_n\right) \\
    &= \delta_{n,nd+1} + \bs\delta_n^\top\left(\alpha\bs{I} - \bs{S}_n\right)^{-1}\bs{s}_n,
\end{align*}
with
\begin{align}
    \notag\delta_{n,nd+1} &:= \frac{\beta}{\lambda_n} + \frac{\lambda_n^\circ}{\lambda_n}\delta_{n, nd+1}' + \frac{\lambda_n^\circ}{\lambda_n}\bs\delta_n'^\top\left(\nu_n\bs{I} - \bs{S}_n\right)^{-1}\bs{s}_n \\
    \label{eq: recu delta_n}\bs{\delta}_n^\top &:= \frac{\lambda_n^\circ}{\lambda_n}\nu_n\bs\delta_n'^\top\left(\nu_n\bs{I} - \bs{S}_n\right)^{-1}= \frac{\lambda_n^\circ}{\lambda_n}\nu_n \left(\bs\delta_{n-1}^\top,\delta_{n-1,(n-1)d+1}\bs\delta^\top \right) \left(\nu_n\bs{I} - \bs{S}_n\right)^{-1}
\end{align}
With a simple calculation, similar to the one done in step 1, one can verify
\begin{equation*}
    \delta_{n,nd+1} = 1 - \bs\delta_n^\top\bs{1},
\end{equation*}
so that $\bs\delta_n^\top$ is a valid initial distribution. We conclude that, conditional on $N(0) = n$, $\bar{Y}(T_\beta)$ is a phase-type random variable characterized as $\mathbb{P}{\rm h}_{nd}(\bs\delta_n, \bs{S}_n)$. 

From the above, the parameters of $\mathbb{P}{\rm h}_{nd}(\bs\delta_n, \bs{S}_n)$ can be identified. 
Concretely, 
there is a straightforward solution to the recursion found in \eqref{eq: recu S_n}, which is given by the following $nd\times nd$ matrix:
\begin{equation}\label{eq: matrix S_n}
    \bs{S}_n := \begin{pmatrix}
        \bs{S} & \bs{s}_1\bs\delta^\top & \dots & \vdots & \vdots \\
        \bs{0}_{d\times d} & \bs{S} & & \bs{s_{n-2}\delta}^\top& \bs{s_{n-1}\delta}^\top \\
        \vdots &&\ddots & \vdots & \vdots \\
        \bs{0}_{d\times d}&\bs{0}_{d\times d}&& \bs{S} & \vdots\\ 
        \bs{0}_{d\times d}&\bs{0}_{d\times d}&\dots & \bs{0}_{d\times d} & \bs{S}
    \end{pmatrix}
\end{equation}

We summarize our findings in the following theorem.

\begin{theorem}\label{theorem: phase-type running max} Suppose $Y(\cdot)\in{\mathscr L}[m,{\bs \lambda},{\mathscr U}(\cdot)\,{\bs 1}^{\top}, {\bs r}]$ with ${\mathscr U}(\cdot)$ satisfying \eqref{eq:LST PH}.
    For given $\beta>0$ and $n\in\{1\dots,m\}$, when conditioning on $N(0)=n$, we have that $\bar{Y}(T_\beta) \sim {\mathbb P}{\rm h}_{nd}({\bs\delta}_n,{\bs S}_n)$, where $\bs\delta_{n}$ can be derived inductively from \eqref{eq: recu delta_n}, and $\bs{S}_n$ is given by \eqref{eq: matrix S_n}.
\end{theorem}

In the remainder of this section we identify the exact asymptotics of $p_m(u, \beta)$, i.e., we find an explicit function $\Psi(\cdot)$ such that $p_m(u,\beta)/\Psi(u)\to 1$ as $u\to\infty$. We determine the exact asymptotics of $p_m(u,\beta)$ in two steps, eventually leading to the result  presented in Theorem \ref{th:PH_AS}.

\medskip

{\it --- Step 1.} We start by recalling a number of known properties of phase-type distributions. In the first place, all eigenvalues $\mu_i$ of ${\bs S}$ have a negative real part \cite[Corollary 3.1.15]{BFN}. Also, it is an immediate consequence of the Perron-Frobenius theorem that the eigenvalue with the largest real part, which we let without loss of generality be $\mu_1$, has no imaginary part. Following the setup of e.g.\ \cite[Eqn.\ (4.3)]{BFN}, we let there be $D$ distinct eigenvalues, which we order such that their real parts are non-decreasing, and let $d_i$ be the multiplicity of the $i$-th eigenvalue (such that $\sum_{i=1}^D d_i=d$). 

As we have seen, $\bs{S}_m$ is an upper triangular block matrix, with the matrix $\bs{S}$ appearing $m$ times on the diagonal. We therefore have that $\bs{S}$ and $\bs{S}_m$ have the same eigenvalues, but the multiplicity of any specific eigenvalue of $\bs{S}_m$ is $m$ times the multiplicity of that eigenvalue in $\bs{S}$. Therefore, letting $d_{m,i}$ denote the multiplicity of the $i$-th eigenvalue of $\bs{S}_m$, we have that $d_{m,i} = md_i$.

\medskip

{\it --- Step 2.}
The result of Theorem \ref{theorem: phase-type running max} states that, conditional on $N(0)=m$, $\bar{Y}(T_\beta)\sim{\mathbb P}_{md}({\bs \delta}_m(\beta),{\bs S}_m)$. Therefore, using the Jordan normal form of $\bs{S}_m$,  the density of the running maximum reads
\[{\mathbb P}_m(\bar{Y}(T_\beta)\in{\rm d}u) =\sum_{i=1}^D \sum_{j=1}^{md_i} c_{i,j}(\beta) \,e^{\mu_i u}\frac{(-\mu_i)^j u^{j - 1}}{(j-1)!}\,{\rm d}u,\]
for certain constants $c_{i,j}(\beta)$. Isolating the dominant term and denoting $\mu:=-\mu_1>0$ gives
\begin{equation*}
    c_{1,md_1}(\beta) \,e^{-\mu u}\frac{\mu^{md_1}u^{md_1-1}}{(md_1-1)!}\,{\rm d}u.
\end{equation*}
Defining $\Psi_j(u):= e^{-\mu u}\,u^{jd_1-1}$, we finally obtain
\begin{equation*}
    p_m(u,\beta) = {\mathbb P}_m\left(\bar{Y}(T_\beta)\geqslant u\right)=c_{1,md_1}(\beta)\frac{\mu^{md_1 - 1}}{(md_1-1)!} \Psi_m(u)+ o(\Psi_m(u)),
\end{equation*}
as $u\to\infty$. 
Upon combining the above steps, we have thus obtained the following result.

\begin{theorem} \label{th:PH_AS} Suppose $Y(\cdot)\in{\mathscr L}[m,{\bs \lambda},{\mathscr U}(\cdot)\,{\bs 1}^{\top}, {\bs r}]$  with ${\mathscr U}(\cdot)$ satisfying \eqref{eq:LST PH}. 
    As $u\to\infty$,
    \[\frac{p_m(u,\beta)}{\Psi_m(u)} \to c_{1,md_1}(\beta)\,\frac{\mu^{md_1}}{(md_1-1)!}.\]
\end{theorem}
This result informally says that the tail distribution of $\bar Y(T_\beta)$ asymptotically behaves, up to a multiplicative constant, as the tail distribution of an Erlang random variable with shape parameter $md_1$ and scale parameter $\mu.$ Indeed, the above result gives rise to the approximation
\[p_m(u,\beta)\approx  c_{1,md_1}(\beta)e^{-\mu u}\,\frac{\mu^{md_1} u^{md_1-1}}{(md_1-1)!}\]
for $u$ large.

\medskip

Figure \ref{fig: asymp erl} gives an impression of the typical differences between exact, asymptotic and Monte Carlo based estimated values of the ruin probability. For the results shown, we evaluated the model with Erlang distributed claims, arrival rates $\lambda^\circ_n = n$, and only positive drifts $r_n=\frac{n}{100}$ (with $n\in\{0,\dots,m\}$). The claim sizes are independent and Erlang distributed with parameters $K=2$ and $\mu=1$. The exact values are determined by applying Theorem \ref{theorem: phase-type running max}, the asymptotic values by applying Theorem \ref{th:PH_AS}, and for the Monte Carlo based estimated values we used large numbers of sample paths, ranging from 200\,000 for smaller values of $u$, to 1\,000\,000 for larger values of $u$.

\begin{figure}[ht]
  \centering
    \includegraphics{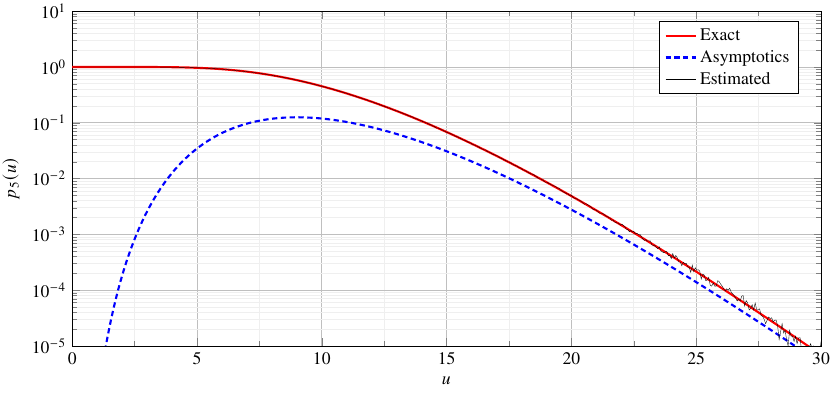}
    \caption{\label{fig: asymp erl}Exact, asymptotic and Monte Carlo based estimated values of the ruin probability, for $m=5$ and Erlang distributed claim sizes, in the model with only positive drifts.}
\end{figure}

\section*{V. Regularly varying claims}
Where in the previous section we considered claims to be of phase-type, a quintessential class of light-tailed distributions, in this section the focus lies on claims that are assumed to be heavy-tailed. Specifically, the class of distributions we work with, characterized via their {\it regularly-varying} tail, is defined as follows \cite{BGT}.

Define, for a non-negative random variable $U$, the coefficients\[b_i:= \lim_{\alpha\downarrow 0} \frac{{\rm d}^i}{{\rm d}\alpha^i} {\mathbb E}\,e^{-\alpha U},\]
such that $(-1)^ib_i$ can be interpreted as the $i$-th moment of $U$ (if it exists). 
We say that a function $L(\cdot)$ is {\it slowly varying} if for any $t>0$ we have that $L(tx)/L(x)\to 1$ as $x\to\infty$; examples of slowly-varying functions are constants and the (iterated) logarithm.
Consider $\delta\in{\mathbb R}^+\setminus {\mathbb N}$ and $\theta>0$, and let ${\mathfrak n}_\delta:=\lfloor \delta\rfloor.$ 
Then we say that a non-negative random variable $U$ is regularly-varying, throughout this section denoted by $U\sim {\mathbb R}{\rm V}_\delta(\theta, L)$, if its LST   can be expressed in terms of the `truncated power series':
\begin{equation}\label{eq:RV}{\mathscr U}(\alpha):={\mathbb E}\,e^{-\alpha U} = \sum_{i=0}^{{\mathfrak n}_\delta} b_i \frac{\alpha^i}{i!} + \theta \alpha^\delta\,L(1/\alpha)+o(\alpha^\delta),\end{equation}
as $\alpha \downarrow 0$, for a slowly-varying function $L(\cdot)$ and constants $b_0,\ldots,b_{{\mathfrak n}_\delta}$. Here $b_i$ represents the $i$-th moment of $B$ (where, evidently, $b_0=1$).

The key result we rely on in this section is a so-called {\it Tauberian theorem} \cite[Theorem 8.1.6]{BGT}, relating the LST's shape around $0$ to the tail behavior of the underlying random variable. It entails that $U\sim {\mathbb R}{\rm V}_\delta(\theta,L)$ is {equivalent} to 
\begin{align}\label{eq:TAIL}\lim_{x\to\infty}{\mathbb P}(U>x) \frac{x^\delta}{L(x)}= \theta\,\frac{(-1)^{{\mathfrak n}_\delta}}{\Gamma(1-\delta)}\end{align}
(where, as an aside, we mention that $(-1)^{{\mathfrak n}_\delta}/\Gamma(1-\delta)$ is positive). A crucial element in this result is the {\it equivalence} between \eqref{eq:RV} and \eqref{eq:TAIL}: we can derive the tail from the shape of the transform around 0 and vice versa. As an aside we mention that the case of integer-valued $\delta$ requires a separate, more delicate analysis, which we leave out here. 

An example of a random variable $X$ with a regularly varying tail is given by 
\begin{equation}\label{eq: RV example}
    {\mathbb P}(X>x) = \frac{C^\epsilon}{(C+x)^\epsilon},\quad x\geqslant 0, 
\end{equation}
for `tail index' $\epsilon\in {\mathbb R}\setminus {\mathbb N}$ and $C>0$; this class of distributions is contained in the larger family of `Pareto-type distributions', and is sometimes referred to as the Lomax distribution. 
In this case $\delta = \epsilon$ and $\theta = \Gamma(1-\epsilon)(-1)^{{\mathfrak n}_\epsilon}\,C^\epsilon.$
Observe that the $k$-th moment of $X$ (for $k\in{\mathbb N}$) exists if $\epsilon> k$. 

\medskip

For the remainder of this section we assume that the claim sizes are independent and identically distributed regularly-varying random variables, so $B\sim {\mathbb R}{\rm V}_\delta(\theta,L)$ for some $\delta\in{\mathbb R}^+\setminus {\mathbb N}$, $\theta>0$, and slowly varying function $L(\cdot)$. It means that each of the $B_i$ has a transform ${\mathscr U}(\cdot)$ that satisfies \eqref{eq:RV}. 
In the previous section we could prove that for phase-type claims all overshoots were phase-type too. In this section we establish a similar property: for regularly varying claims we show that the running maximum is also regularly varying. From there we can determine the asymptotics of $p_m(u,\beta)$ and $p_m(u)$ in the regime that $u$ grows large. Like in the previous section, we consider the situation in which the small clients' L\'evy processes correspond to deterministic drifts, i.e., we have that $Y(\cdot)\in{\mathscr L}[m,{\bs \lambda},{\mathscr U}(\cdot)\,{\bs 1}^{\top}, {\bs r}]$.

The main objective of our analysis is to  derive that the LST of the running maximum $\bar Y(T_\beta)$ has a regularly varying distribution. We do so by appealing to the recursion given in Theorem \ref{T1}, in combination with the Tauberian theorem mentioned above. 
\begin{lemma}\label{lemma: pi_m in case RV}
    For any $\beta > 0$ and  $n\in\{1,\dots,m\}$, we have that
    \begin{equation}\label{eq: pi_m in case RV}
        \pi_n(\alpha,\beta) = \sum_{i=0}^{\mathfrak{n}_\delta} p_{n,i}(\beta)\frac{\alpha^i}{i!} + \Phi_n(\beta)\,\alpha^\delta\, L(1/\alpha) + o(\alpha^\delta)
    \end{equation}
    as $\alpha\downarrow0$ ,
    where  \[\Phi_n(\beta) = \theta \sum_{j=1}^n\prod_{i=j}^n \frac{\lambda_i^\circ}{\lambda_i};\] 
    here, $p_{n,0}(\beta) = 1$ and  $p_{n,1}(\beta),\ldots,p_{n,\mathfrak{n}_\delta}(\beta)$ are constants.
\end{lemma}
\begin{proof}
We prove \eqref{eq: pi_m in case RV} by induction. 

\medskip 

{\it --- Step 1.} For $n=1$, from Theorem \ref{T1}, we have that 
\begin{align*}
    \pi_1(\alpha,\beta) &= \frac{\beta}{\lambda_1} + \frac{\lambda_1^\circ}{\lambda_1}\frac{\nu_1}{\nu_1 - \alpha}\left(\mathscr{B}(\alpha) - \frac{\alpha}{\nu_1}\mathscr{B}(\nu_1)\right) \\
    &= \frac{\beta}{\lambda_1} + \frac{\lambda_1^\circ}{\lambda_1}\left(1 + \frac{\alpha}{\nu_1} + o(\alpha)\right)\left(\sum_{i=0}^{{\mathfrak n}_\delta} b_i \frac{\alpha^i}{i!} + \theta\alpha^\delta \,L(1/\alpha) + o(\alpha^\delta) - \frac{\alpha}{\nu_1}\mathscr{B}(\nu_1)\right) \\
    &= \sum_{i=0}^{\mathfrak{n}_{\delta}} p_{1,i}(\beta) \frac{\alpha^i}{i!} + \theta\,\frac{\lambda_1^\circ}{\lambda_1} \alpha^\delta L(1/\alpha) + o(\alpha^\delta),
\end{align*}
where we remark that $\theta\,\lambda_1^\circ/\lambda=\Phi_1(\beta)$. Note that
\begin{equation}
    p_{1,0}(\beta) = \frac{\beta}{\lambda_1} + \frac{\lambda_1^\circ}{\lambda_1}\cdot 1 \cdot b_0 =  \frac{\beta}{\lambda_1^\circ+\beta} + \frac{\lambda_1^\circ}{\lambda_1^\circ+\beta} = 1,
\end{equation}
where we use that  $b_0 = 1$.

\medskip

{\it --- Step 2.} By the induction hypothesis we may assume, for any $n\in\{1,\dots,m\}$,  that for all $k<n$ we have that $\pi_{k}(\alpha,\beta)$ is of the form
\begin{align*}
     \sum_{i=0}^{{\mathfrak n}_\delta} p_{k,i}(\beta) \frac{\alpha^i}{i!} + \Phi_k(\beta)\,\alpha^\delta \,L(1/\alpha) + o(\alpha^\delta),
\end{align*}
with $p_{k,0}(\beta) = 1$. The next step is to use the recursion given in  Theorem \ref{T1}: we find for $\pi_n(\alpha,\beta)$, with $\nu_n = \lambda_n/r_n$:
\begin{align*}
    \pi_n(\alpha,\beta) &= \frac{\beta}{\lambda_n} + \frac{\lambda_n^\circ}{\lambda_n}\left(1 + \frac{\alpha}{\nu_n} + o(\alpha)\right) \cdot\Bigg({\mathscr P}_n^{(1)}(\alpha,\beta)\,{\mathscr P}_n^{(2)}(\alpha,\beta) -\frac{\alpha}{\nu_n}\mathscr{B}(\nu_n)\,\pi_{n-1}(\nu_n,\beta)\Bigg),
\end{align*}  
where
\begin{align*}
    {\mathscr P}_n^{(1)}(\alpha,\beta)&:=\sum_{i=0}^{{\mathfrak n}_\delta} b_i \frac{\alpha^i}{i!} + \theta\alpha^\delta \,L(1/\alpha) + o(\alpha^\delta),\\
    {\mathscr P}_n^{(2)}(\alpha,\beta) &:= \sum_{i=0}^{{\mathfrak n}_\delta} p_{n-1,i}(\beta) \frac{\alpha^i}{i!} + \Phi_{n-1}(\beta)\,\alpha^\delta \,L(1/\alpha) + o(\alpha^\delta).
\end{align*}
By collecting terms that are of the same order, it is readily verified that $\pi_n(\alpha,\beta)$ can be written in the form
    \[ \sum_{i=0}^{\mathfrak{n}_\delta} p_{n,i}(\beta) \frac{\alpha^i}{i!} + \frac{\lambda_n^\circ}{\lambda_n}\big(\,b_0\,\Phi_{n-1}(\beta) + p_{n-1,0}(\beta)\,\theta\,\big)\,\alpha^\delta L(1/\alpha) + o(\alpha^\delta).
\]
for some coefficients $p_{n,i}(\beta)$. Now, since $b_0 = p_{n-1,0}(\beta) = 1$, we find
\begin{align*}
    p_{n,0}(\beta) &= \frac{\beta}{\lambda_n} + \frac{\lambda_n^\circ}{\lambda_n}\cdot 1 \cdot b_0 \cdot p_{n-1,0}(\beta) = \frac{\beta}{\lambda_n} + \frac{\lambda_n^\circ}{\lambda_n} = 1.
\end{align*}
Furthermore, 
\begin{align*}
    \frac{\lambda_n^\circ}{\lambda_n}\left(\,b_0\,\Phi_{n-1}(\beta) + p_{n-1,0}(\beta)\,\theta\,\right) &= \frac{\lambda_n^\circ}{\lambda_n}\left(\theta \sum_{j=1}^{n-1}\prod_{i=j}^{n-1} \frac{\lambda_i^\circ}{\lambda_i} + \theta\right) = \Phi_n(\beta),
\end{align*}
where verification of the last equality requires a few elementary calculations. 
Upon combining the above findings, we conclude
\begin{align*}
     \pi_n(\alpha,\beta) = \sum_{i=0}^{{\mathfrak n}_\delta} p_{n,i}(\beta) \frac{\alpha^i}{i!} + \Phi_n(\beta)\,\alpha^\delta \,L(1/\alpha) + o(\alpha^\delta),
\end{align*}
with $p_{n,0}(\beta) = 1$.
We have thus proven the claim. 
\end{proof}

Recall the Tauberian theorem that we discussed above, i.e., the {equivalence} between the statements \eqref{eq:RV} and \eqref{eq:TAIL}.
The conclusion is that $\bar{Y}(T_\beta) \sim {\mathbb R}{\rm V}_\delta(\Phi_n(\beta),L)$ (conditional on $N(0)=n$, that is). By the Tauberian theorem we thus conclude that this is equivalent to 
\begin{equation*}
    \lim_{u\to\infty} \mathbb{P}_n(\bar{Y}(T_\beta) > u)\frac{u^\delta}{L(u)} = \Phi_n(\beta) \frac{(-1)^{\mathfrak{n}_\delta}}{\Gamma(1-\delta)},
\end{equation*}
for any $n$. Specializing to the case of $n=m$, this brings us to the following result.
\begin{theorem}
    Suppose $Y(\cdot)\in{\mathscr L}[m,{\bs \lambda},{\mathscr U}(\cdot)\,{\bs 1}^{\top}, {\bs r}]$  with ${\mathscr U}(\cdot)$ satisfying \eqref{eq:RV}. As $u\to\infty$,
    \begin{equation}
        p_m(u,\beta) \frac{u^\delta}{L(u)} \to \Phi_m(\beta) \frac{(-1)^{\mathfrak{n}_\delta}}{\Gamma(1-\delta)}.
    \end{equation}
\end{theorem}

\begin{remark}\label{R2}
    {\em The above theorem gives rise to the approximation, in the regime that $u$ is large,
\[p_m(u,\beta) \approx \frac{\Phi_m(\beta)}{\theta}\,{\mathbb P}(B>u)= \left(\sum_{j=1}^m\prod_{i=j}^m \frac{\lambda_i^\circ}{\lambda_i}\right)\,{\mathbb P}(B>u).\]
    Interestingly, the ruin probability inherits the tail behavior of the claims, in that our result reveals that the ruin probability is asymptotically proportional to the tail distribution of the claim sizes. This is essentially different from what happens in the model with an infinite customer-pool, where the tail index of  the ruin probability is one lower than the one of the claim sizes, i.e., the tail of the ruin probability is one degree heavier than the tail of the claim-size distribution; see e.g.\ the textbook accounts~\cite[Theorem 2.1]{AA} or \cite[Theorem 8.4]{DM}. }\hfill$\Diamond$
  \end{remark}

\begin{remark}\label{R3}
    {\em   In relation to Remark \ref{R2}, it is observed that, for any $j\in\{1,\ldots,m\}$, 
    \[\prod_{i=j}^m \frac{\lambda_i^\circ}{\lambda_i} = {\mathbb P}(M\geqslant m-j+1),\]
    where, as before, $M$ denotes the number of large claims before killing. Consequently, we have that, recalling that $M$ attains values in $\{0,1,\ldots,m\}$, 
    \[\sum_{j=1}^m\prod_{i=j}^m \frac{\lambda_i^\circ}{\lambda_i}= \sum_{j=1}^m\ {\mathbb P}(M\geqslant m-j+1) = \sum_{j=1}^m {\mathbb P}(M\geqslant j)={\mathbb E}\,M.\]
    In other words, we find the approximation
    \[p_m(u,\beta) \approx {\mathbb E}M\cdot{\mathbb P}(B>u); \]
    cf.\ for instance \cite[Lemma 2.2]{MB} or \cite[Lemma X.2.2]{ASM2}. This approximation is particularly appealing, and should be seen in relation to the fact \cite[Proposition X.1.7]{ASM2} that for our regularly varying random variables $B_1,\dots, B_n$ we have that ${\mathbb P}(B_1+\cdots+B_n>u)/{\mathbb P}(B_1>u)\to n$ as $u\to\infty$.
    \hfill$\Diamond$
    }
\end{remark}

Figure \ref{fig: asymp RV} shows a comparison between Laplace inverted, asymptotic and Monte Carlo based estimated values of the ruin probability for regular varying distributed claims. We worked with an instance of the model in which $\lambda^\circ_n = n$ and $r_n= n^2$. The claim sizes are independent and regularly varying random variables, distributed according to the cumulative distribution function given in \eqref{eq: RV example}, with parameters $C=1$ and $\epsilon=\frac{3}{2}$. The Laplace inverted values are derived numerically by applying Theorem \ref{T1} and Lemma~\ref{lemma: relation pi_n rho_n} for $\beta = 0$, where we approximated the LSTs $\mathscr{B}_n(\alpha)$ given by
\[\int_0^\infty e^{-\alpha u}\,\frac{3}{2}\frac{1}{(1+u)^{5/2}}\,\mathrm{d}u\]
numerically for the specific values of $\alpha$ needed for Stehfest's inversion algorithm. The asymptotic values are derived with the approximation in Remark \ref{R3}, and for the Monte Carlo based estimated values we used 20\,000 sample paths per $u$. 

\begin{figure}[ht]
    \centering
    \includegraphics{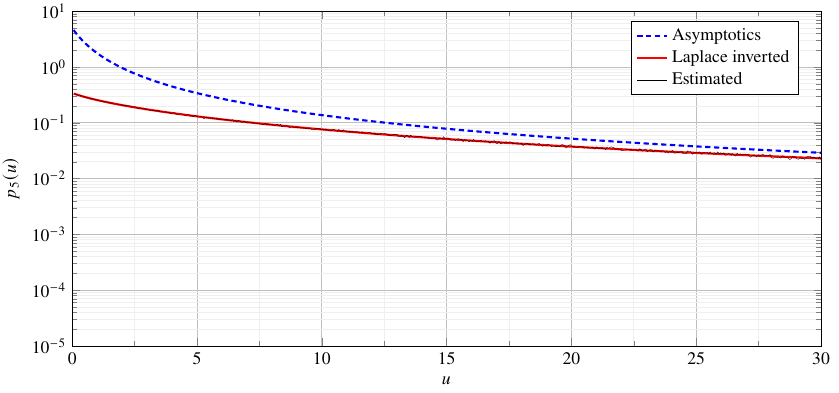}
    \caption{\label{fig: asymp RV}Laplace inverted, asymptotic and Monte Carlo based estimated values of the ruin probability, for $m=5$ and regular varying distributed claim sizes, in the model with only positive drifts.}
\end{figure}

\section*{VI. Overshoot analysis}
So far we essentially concentrated on the distribution of $\bar Y(T_\beta)$, which facilitated the identification of the asymptotics of $p_m(u,\beta)$ for large $u$. In the present section
we provide an exact analysis of more detailed quantities, that do not only incorporate whether level $u$ has been reached or not (prior to the exponential clock $T_\beta$, that is), but also describe by how much $u$ has been exceeded. We concentrate on the case that the small clients correspond to determinsitic drifts, i.e., $Y(\cdot)\in {\mathscr L} [m, {\bs \lambda}, {\bs {\mathscr B}}(\cdot),{\bs r}]$.

We define, for $n\in\{1,\ldots,m\}$ and $k\in\{0,\ldots,m-1\}$,
\[\eta_{n,k}(u,\alpha,\beta):= {\mathbb E}_n \Big(e^{-\alpha (Y(\tau(u))-u)}\,{\bs 1}_{\{N(\tau(u))=k,\tau(u)\leqslant T_\beta\}}\Big);\]
the LST of the overshoot over level $u$, intersected with the indicator function of the event that $u$ is reached before killing by the claim of the $(n-k)$-th client. 
In the present section we uniquely characterize the object $\eta_{n,k}(u,\alpha,\beta)$, and thus identify the distribution of the overshoot over level~$u$. Our findings are then used to develop an alternative characterization of $\pi_m(\alpha,\beta)$, i.e., a representation different from the one we gave in Theorem \ref{T1}.

\begin{remark} \em 
In this section we do consider the setting that the small clients are modelled as deterministic drifts, rather than the more general class of spectrally-positive L\'evy processes. The main reason for this is that in the analysis of the present section an important role is played by the overshoot over level $0$ (i.e., the random variable $Y(\tau(0))$). More specifically, we write $\bar Y(T_\beta)$ as the sum of such ladder heights, which can be a `degenerate notion' for general spectrally-positive L\'evy processes; for instance for Brownian motion $Y(\tau(0))$ equals $0$ almost surely. \hfill$\Diamond$ 
\end{remark}

The following recursion plays a key role in our analysis. 

\begin{lemma}\label{lemma: psi_mk}     Suppose $Y(\cdot)\in {\mathscr L} [m, {\bs \lambda}, {\bs {\mathscr B}}(\cdot),{\bs r}]$. 
   For any $u>0$, $\alpha\geqslant 0$, and for $n\in\{1,\ldots,m\}$ and $k\in\{0,\ldots,n-2\}$, 
    \begin{align}\label{eq: recursion psi k<m-1}
     \eta_{n,k}(u,\alpha,\beta)=&\:\lambda_n^\circ \int_0^\infty e^{-\lambda_n t}\int_0^{u + r_nt} \,\mathbb{P}(B_{m-n+1} \in {\mathrm d}s)\,\eta_{n-1,k}(u +r_nt- s, \alpha,\beta)\,{\mathrm d}t,
    \end{align}
    while for $n\in\{1,\ldots,m\}$,
    \begin{align}\label{eq: recursion psi k=m-1}
     \eta_{n,n-1}(u,\alpha,\beta)=&\:\lambda_n^\circ  \int_0^\infty e^{-(\lambda_n - \alpha r_n)t} \int_{u + r_nt}^\infty e^{-\alpha (s-u)}\,\mathbb{P}(B_{m-n+1}\in \mathrm{d}s)  \, {\mathrm d}t.
    \end{align}
\end{lemma}
\begin{proof}
We start by considering the (somewhat more straightforward) case that $k=n-1$.
Observe that, to make sure $\tau(u)\leqslant T_\beta$, first claim should arrive before the `exponential clock' $T_\beta$ expires, which happens with probability $\lambda_n^\circ /\lambda_n$. Then the time till the claim is exponentially distributed with parameter $\lambda_n$. Given this time is $t$, the claim should be larger than $u+r_nt$ to make sure the net cumulative claim process exceed $u$; if it has size $s\in(u+r_nt,\infty)$, we have that the overshoot $Y(\tau(u))-u$ equals $s-r_nt-u$. We thus obtain
\[\frac{\lambda_n^\circ }{\lambda_n}\int_0^\infty\lambda_n e^{-\lambda_n t} \int_{u+r_n t}^\infty e^{-\alpha(s-r_nt-u)}\,{\mathbb P}(B_{m-n+1}\in {\rm d}s)\,{\rm d}t,\]
equalling \eqref{eq: recursion psi k=m-1}.

The other scenario can be dealt with in essentially the same way. Now the claim size $s$ is in the interval $[0,u+r_nt)$, so that still a distance $u+r_nt-s$ still needs to be bridged in order to exceed level $u$ (which needs to be done with $n-1$ rather than with $n$ claims). This leads to the expression
\[\frac{\lambda_n^\circ }{\lambda_n}\int_0^\infty\lambda_n e^{-\lambda_n t} \int_0^{u+r_n t} \,{\mathbb P}(B_{m-n+1}\in {\rm d}s)\,\eta_{n-1,k}(u+r_nt-s,\alpha,\beta)\,{\rm d}t,\]
which is easily verified to reduce to \eqref{eq: recursion psi k<m-1}. \end{proof}

The equation in \eqref{eq: recursion psi k<m-1} corresponds to the contribution by the scenario that the net cumulative claim process remains below $u$ after the first claim arrival, whereas \eqref{eq: recursion psi k=m-1} represents the contribution by the scenario in which the level $u$ is already reached by the first claim.

The next objective is to use Lemma \ref{lemma: psi_mk} to set up, for $n\in\{1,\ldots,m\}$ and $k\in\{0,\ldots,n-1\}$, an iteration scheme for the {\it double} transform
\[\xi_{n,k}(\alpha,\beta,\gamma):= \int_0^\infty \gamma e^{-\gamma u}\,\eta_{n,k}(u,\alpha,\beta)\,{\rm d}u.\]
We incorporate a factor $\gamma$ in the integrand so that the object can be interpreted as the counterpart of $\eta_{n,k}(u,\alpha,\beta)$, but with the deterministic surplus level $u$ replaced by a random counterpart. Indeed, we can rewrite
\[\xi_{n,k}(\alpha,\beta,\gamma):= {\mathbb E}\, \eta_{n,k}(U_\gamma,\alpha,\beta),\]
with $U_\gamma$ exponentially distributed with parameter $\gamma$, drawn independently of the net cumulative claim process $Y(t)$ and the killing time $T_\beta.$

\begin{lemma}\label{lemma: recu xi}     Suppose $Y(\cdot)\in {\mathscr L} [m, {\bs \lambda}, {\bs {\mathscr B}}(\cdot),{\bs r}]$.  For any $\alpha\geqslant 0$, $\beta>0$, and $\gamma\geqslant 0$, and for $n\in\{1,\ldots,m\}$ and $k\in\{0,\ldots,n-2\}$, 
\begin{equation} \label{eq: recursion xi}
     \xi_{n,k}(\alpha,\beta,\gamma)=\frac{\lambda_n^\circ }{\lambda_n-\gamma r_n} \big(
  {\mathscr B}_{m-n+1}(\gamma)\,\xi_{n-1,k}(\alpha,\beta,\gamma) - \frac{\gamma}{\nu_n}{\mathscr B}_{m-n+1}(\nu_n)\,\xi_{n-1,k}(\alpha,\beta,\nu_n)
 \big).
\end{equation}
while for $n\in\{1,\ldots,m\}$,
\begin{equation}\xi_{n,n-1}(\alpha, \beta, \gamma)=\frac{\lambda_n^\circ \,\gamma}{\lambda_n-\alpha r_n}\left(\frac{{\mathscr B}_{m-n+1}(\gamma)-{\mathscr B}_{m-n+1}(\alpha)}{\alpha-\gamma}-\frac{{\mathscr B}_{m-n+1}(\gamma)-{\mathscr B}_{m-n+1}(\nu_n)}{\nu_n-\gamma}\right).\label{eq: initial xi}\end{equation}
\end{lemma}
\begin{proof}
We start by evaluating $\xi_{n,k}(\alpha,\beta,\gamma)$ for $k=n-1$, corresponding to the scenario dealt with in \eqref{eq: recursion psi k=m-1}, which is the easier case. We have to change the integration order such that the `easy' integrations (the ones corresponding to $t$ and $u$, that is) are to be performed first. Following this procedure we obtain
\begin{align}
\notag\xi_{n,n-1}(\alpha, \beta, \gamma) &= \lambda_n^\circ  \int_0^\infty \,\gamma e^{-\gamma u}\int_0^\infty e^{-(\lambda_n - \alpha r_n)t} \int_{u + r_nt}^\infty e^{-\alpha(s-u)}\,\mathbb{P}(B\in \mathrm{d}s) \, {\mathrm d}t\,{\rm d}u\\
&=\lambda_n^\circ \,\gamma\int_{s=0}^\infty e^{-\alpha s} \,{\mathbb P}(B_{m-n+1}\in {\rm d}s)\int_{u=0}^s e^{-(\gamma-\alpha)u}\int_{t=0}^{(s-u)/r_n} e^{-(\lambda_m - \alpha r_n)t} \,{\rm d}t\,{\rm d}u\notag \\
&= \frac{\lambda_n^\circ \,\gamma}{\lambda_n-\alpha r_n}\int_{s=0}^\infty e^{-\alpha s} \,{\mathbb P}(B_{m-n+1}\in {\rm d}s)\int_{u=0}^s e^{-(\gamma-\alpha)u}\left(1- e^{-(\lambda_n-\alpha r_n)(s-u)/r_n}\right)\,{\rm d}u \notag\\
&= \frac{\lambda_n^\circ \,\gamma}{\lambda_n-\alpha r_n}\int_{s=0}^\infty e^{-\gamma s} \,{\mathbb P}(B_{m-n+1}\in {\rm d}s)\int_{u=0}^s e^{-(\alpha-\gamma)u}\left(1- e^{-(\lambda_n-\alpha r_n)u/r_n}\right)\,{\rm d}u \notag\\
&=\frac{\lambda_n^\circ \,\gamma}{\lambda_n-\alpha r_n}\int_{s=0}^\infty e^{-\gamma s} \,{\mathbb P}(B_{m-n+1}\in {\rm d}s)\int_{u=0}^s \left(e^{-(\alpha-\gamma)u}- e^{-(\lambda_n/r_n-\gamma)u}\right)\,{\rm d}u \notag\\
&=\frac{\lambda_n^\circ \,\gamma}{\lambda_n-\alpha r_n}\int_{s=0}^\infty e^{-\gamma s} \,{\mathbb P}(B_{m-n+1}\in {\rm d}s) \left(\frac{1-e^{-(\alpha-\gamma)s}}{\alpha-\gamma}-\frac{1-e^{-(\lambda/r_n-\gamma)s}}{\lambda_n/r_n-\gamma}\right), \notag
\end{align}
which corresponds to \eqref{eq: initial xi}.

In case $k \in\{0,\ldots,n-2\}$ the computation of $\xi_{n,k}(\alpha,\beta,\gamma)$ becomes somewhat more involved. Performing the change of variable $w:=u+r_nt$ and again changing the order of integration,
\begin{align}\notag
 \xi_{n,k}(\alpha,\beta,\gamma)&=\lambda_n^\circ \int_0^\infty \,\gamma e^{-\gamma u}\int_0^\infty e^{-\lambda_n t}\int_0^{u + r_nt} \,\mathbb{P}(B_{m-n+1} \in {\mathrm d}s)\,\eta_{n-1,k}(u +r_nt- s, \alpha,\beta)\,{\mathrm d}t\,{\rm d}u\\
 &=\frac{\lambda_n^\circ \,\gamma}{r_n}\int_{s=0}^\infty{\mathbb P}(B_{m-n+1}\in{\rm d}s)\int_{w=s}^\infty e^{-(\lambda_n/r_n)w}\eta_{n-1,k}(w-s,\alpha,\beta)\int_{u=0}^w e^{-(\gamma-\lambda/r_n)u}\,{\rm d}u\,{\rm d}w\notag \\
 &=\frac{\lambda_n^\circ \,\gamma}{\lambda_n-\gamma r_n} \int_{s=0}^\infty{\mathbb P}(B_{m-n+1}\in{\rm d}s)\int_{w=s}^\infty \eta_{n-1,k}(w-s,\alpha,\beta)\left(
 e^{-\gamma w} - e^{-(\lambda_n/r_n)w}
 \right)\,{\rm d}w\notag \\
 &=\frac{\lambda_n^\circ \,\gamma}{\lambda_n-\gamma r_n} \int_{s=0}^\infty{\mathbb P}(B_{m-n+1}\in{\rm d}s)\Bigg(\int_0^\infty e^{-\gamma(s+w)}\eta_{n-1,k}(w,\alpha,\beta)\,{\rm d}w 
 \:-\notag\\
 &\hspace{5.62cm}\int_0^\infty e^{-(\lambda_n/r_n)(s+w)}\eta_{n-1,k}(w,\alpha,\beta) \,{\rm d}w
 \Bigg)\notag,
\end{align}
in which, after noting that the two double integrals in the last expression factorize, we immediately recognize \eqref{eq: recursion xi}.
\end{proof}

\begin{remark}
  {\em The expressions in \eqref{eq: recursion xi} and \eqref{eq: initial xi} contain fractions of which the denominator can attain the value $0$. It is observed though that for those values the numerator vanishes as well, and that the expressions can be evaluated by a straightforward application of L'H\^opital's rule. In particular,
  \[
  \xi_{n,n-1}(\alpha, \beta, \nu_n)=\frac{\lambda_n^\circ \,\gamma}{\lambda_n-\alpha r_n}\left(\frac{{\mathscr B}_{m-n+1}(\nu_n)-{\mathscr B}_{m-n+1}(\alpha)}{\alpha-\nu_n}+{\mathscr B}'_{m-n+1}(\nu_n)\right)\]
  for $n\in\{1,\ldots,m\}.$} \hfill$\Diamond$ 
\end{remark}

Relying on Lemma \ref{lemma: recu xi}, the algorithm identifies the transforms $\xi_{n,k}(\alpha,\beta,\gamma)$, for all  $n\in\{1,\ldots,m\}$ and $k\in\{0,\ldots,n-1\}$. We have thus uniquely characterized the joint distribution of the overshoot over level $u$ and the index of the client due to whom $u$ has been exceeded. 

\begin{theorem}\label{prop: alg}     Suppose $Y(\cdot)\in {\mathscr L} [m, {\bs \lambda}, {\bs {\mathscr B}}(\cdot),{\bs r}]$. 
For any $k\in\{0,\ldots,m-1\}$, the transform $\xi_{k+1,k}(\alpha,\beta,\gamma)$ follows from \eqref{eq: initial xi}. Subsequently, in case $k\in\{0,\ldots,m-2\}$, the transforms $\xi_{k+2,k}(\alpha,\beta,\gamma),\ldots,\xi_{m,k}(\alpha,\beta,\gamma)$ follow by repeatedly applying the recursion \eqref{eq: recursion xi}.
\end{theorem}

In the remainder of this section the objective is to identify the LST of the running maximum over an exponentially distributed interval, as before denoted by $\pi_m(\alpha,\beta)$. To this end we define the LST pertaining to the overshoot over level $0$:
\[\zeta_{m,k}(\alpha, \beta):= {\mathbb E}_m \Big(e^{-\alpha Y(\tau(0))}\,{\bs 1}_{\{N(\tau(0))=k,\tau(0)\leqslant T_\beta\}}\Big)=\lim_{\gamma\to\infty}\xi_{m,k}(\alpha,\beta,\gamma).\]
Using the results of Lemma \ref{lemma: recu xi} and Theorem \ref{prop: alg}, we can find exact expressions for these $\zeta_{m,k}(\alpha, \beta)$, in terms of the transforms $\xi_{n,k}(\alpha,\beta,\gamma)$. Indeed, for $n \in \{1,\ldots,m\}$ and $k \in \{0,\ldots,n-2\}$ we have
\begin{align}\label{zeta1}
    \zeta_{n,n-1}(\alpha, \beta) &= \frac{\lambda_n^\circ }{r_n} \frac{\mathscr{B}_{m-n+1}(\nu_n) - \mathscr{B}_{m-n+1}(\alpha)}{\alpha-\nu_n}, \\
    \zeta_{n,k}(\alpha, \beta) &= \frac{\lambda_n^\circ }{\lambda_n}{\mathscr B}_{m-n+1}(\nu_n)\,\xi_{n-1,k}(\alpha,\beta,\nu_n).\label{zeta2}
\end{align}
In the following theorem we present a relation that facilitates a recursive procedure to evaluate the LST $\pi_m(\alpha, \beta)$ recursively, using the transforms $\zeta_{n,k}(\alpha, \beta)$.

\begin{theorem}\label{theorem: recu pi_m}     Suppose $Y(\cdot)\in {\mathscr L} [m, {\bs \lambda}, {\bs {\mathscr B}}(\cdot),{\bs r}]$. 
    We can evaluate $\pi_m(\alpha, \beta)$ recursively through, for any $\alpha\geqslant 0$, $\beta>0$,
    \begin{equation}\label{eq:recupi}
        \pi_m(\alpha,\beta) = \mathbb{P}_m(\tau(0) > T_\beta) + \sum_{k=0}^{m-1} \zeta_{m,k}(\alpha,\beta)\,\pi_{k}(\alpha,\beta),
    \end{equation}
    with $\pi_0(\alpha, \beta) = 1$ and
    \begin{equation*}
        \mathbb{P}_m(\tau(0) > T_\beta) = 1 - \sum_{k=0}^{m-1} \zeta_{m,k}(0,\beta).
    \end{equation*}
    The functions $\zeta_{n,k}(\alpha,\beta)$, for $n\in\{1,\ldots,m\}$ and $k\in\{0,\ldots,n-2\}$, follow from \eqref{zeta1}-\eqref{zeta2}. 
\end{theorem}
\begin{proof}
    First, we distinguish between the events of exceeding level 0 before killing and vice versa. When $\tau(0) > T_\beta$, we have that $\bar Y(T_\beta) = 0$, so that
    \begin{align*}
        \pi_m(\alpha, \beta) &= {\mathbb E}_m\left(e^{-\alpha \bar Y(T_\beta)}{\bs 1}_{\{\tau(0) \leqslant T_\beta\}}\right) + {\mathbb E}_m\left(e^{-\alpha \bar Y(T_\beta)}{\bs 1}_{\{\tau(0) > T_\beta\}}\right) \\
        &= {\mathbb E}_m\left(e^{-\alpha \bar Y(T_\beta)}{\bs 1}_{\{\tau(0) \leqslant T_\beta\}}\right) + {\mathbb E}_m\left({\bs 1}_{\{\tau(0) > T_\beta\}}\right)\\&= {\mathbb E}_m\left(e^{-\alpha \bar Y(T_\beta)}{\bs 1}_{\{\tau(0) \leqslant T_\beta\}}\right) + \mathbb{P}_m(\tau(0) > T_\beta).
    \end{align*}
    By the strong Markov property and the memoryless property, conditioning on $N(\tau(0)) = k$ for $k\in\{0,\ldots,m-1\}$, we find
    \begin{align*}
        {\mathbb E}_m\left(e^{-\alpha \bar Y(T_\beta)}{\bs 1}_{\{\tau(0) \leqslant T_\beta, N(\tau(0)) = k\}}\right) &= {\mathbb E}_m\left(e^{-\alpha Y(\tau(0))}{\bs 1}_{\{\tau(0) \leqslant T_\beta, N(\tau(0)) = k\}}\right) {\mathbb E}_k\left(e^{-\alpha \bar Y(T_\beta)}\right) \\
        &= \zeta_{m,k}(\alpha, \beta)\,\pi_k(\alpha,\beta).
    \end{align*}
    Consequently, using that $N(\tau(0))$ takes values in $\{0,\ldots,m-1\}$, 
    \begin{align*}
        {\mathbb E}_m\left(e^{-\alpha \bar Y(T_\beta)}{\bs 1}_{\{\tau(0) \leqslant T_\beta\}}\right) &=  \sum_{k=0}^{m-1} {\mathbb E}_m\left(e^{-\alpha \bar Y(T_\beta)}{\bs 1}_{\{\tau(0) \leqslant T_\beta, N(\tau(0)) = k\}}\right) = \sum_{k=0}^{m-1} \zeta_{m,k}(\alpha, \beta)\,\pi_k(\alpha,\beta).
    \end{align*}
    Upon combining the above, we find for $\pi_m(\alpha,\beta)$ the recursive relation \eqref{eq:recupi}. 
    Finally, as a consequence of the fact that $\pi_k(0, \beta) = 1$ for all $k\in\{0,\ldots,m-1\}$, we have
    \begin{align}\label{P tau Tbeta}
        1 &= \mathbb{P}_m(\tau(0) > T_\beta) + \sum_{k=0}^{m-1} \zeta_{m,k}(0, \beta),
    \end{align}
    from which the expression for $\mathbb{P}_m(\tau(0) > T_\beta)=1-p_m(0,\beta)$ follows.
\end{proof}
By taking the limit of $\beta\downarrow 0$ in \eqref{P tau Tbeta}, we find the probability of the net cumulative process never becoming positive, as stated in the following corollary.
\begin{corollary}     Suppose $Y(\cdot)\in {\mathscr L} [m, {\bs \lambda}, {\bs {\mathscr B}}(\cdot),{\bs r}]$. 
    For any $m\in{\mathbb N}$,
    \begin{align*}
        p_m(0) = \sum_{k=0}^{m-1} \zeta_{m,k}(0,0).
    \end{align*}
\end{corollary}

Besides Theorem \ref{theorem: recu pi_m}, there is an alternative way of expressing $\pi_m(\alpha,\beta)$ in terms of the $\zeta_{n,k}(\alpha,\beta)$ (with $n\in\{0,\ldots,m\}$ and $k\in\{0,\ldots,n-1\}$). In this approach we distinguish between the scenario that the running maximum over the time interval $[0,T_\beta]$ is achieved {\it before} the $m$-th claim arrival and the scenario that it is achieved {\it at} the $m$-th claim arrival.
In line with what we have seen in the proof of Theorem~\ref{theorem: recu pi_m}, using that
$\zeta_{n,k}(0, \beta)= {\mathbb P}_n (N(\tau(0))=k,\tau(0)\leqslant T_\beta)$ we have
\begin{equation}
    \label{eq: last rec}1-\sum_{\ell=0}^{n-1} \zeta_{n,\ell} (0,\beta)= {\mathbb P}_n(\tau(0)>T_\beta).\end{equation}
Combining the above observations, we arrive at the following result, which can be seen as an explicit solution of the recursion stated in Theorem \ref{theorem: recu pi_m}.
\begin{theorem}\label{pr:p2}     Suppose $Y(\cdot)\in {\mathscr L} [m, {\bs \lambda}, {\bs {\mathscr B}}(\cdot),{\bs r}]$. 
  We can   evaluate $\pi_m(\alpha, \beta)$ through, for any $\alpha\geqslant 0$, $\beta>0$,
\begin{align*}\pi_m(\alpha,\beta)=\:&\sum_{m=i_0>i_1>\cdots>i_{j}>0}\left( \prod_{\ell =0}^{j-1} \zeta_{i_\ell,i_{\ell+1}} (\alpha,\beta)\right)\left(1-\sum_{\ell=0}^{i_j-1}\zeta_{i_j,\ell}(0,\beta)\right)\:+\\
&\sum_{m=i_0>i_1>\cdots>i_{j}=0}\left( \prod_{\ell =0}^{j-1} \zeta_{i_\ell,i_{\ell+1}} (\alpha,\beta)\right).
\end{align*}
\end{theorem}

\begin{proof}
    In this proof a crucial role is played by the concept of {\it ladder height}, being the difference between two subsequent record values of the process $Y(t)$. Observe that the running maximum $\bar Y(T_\beta)$ can be expressed as the sum of (maximally $m$) of such ladder heights. The main idea is to work with a partition that keeps track of the values of the $N(t)$ process at all the times $Y(t)$ attains a new record value. 
    The first term of the right hand side captures all scenarios in which $\bar Y(T_\beta)$ is attained before the $m$-th claim arrival (bearing in mind the identity \eqref{eq: last rec}), whereas the second terms concerns all scenarios in which $\bar Y(T_\beta)$ is attained at the $m$-th claim arrival.

    One can alternatively prove this theorem by verifying that the postulated form satisfies the recursion of Theorem \ref{theorem: recu pi_m}.
\end{proof}

\section*{VII. Discussion and concluding remarks}
In this paper, we considered an insurance model with the unconventional feature of having a finite number of major clients, in addition to a large pool of small clients. This model thus represents a significant deviation from the standard Cramér-Lundberg model, which assumes an infinitely large pool of essentially homogeneous clients, justifying the use of Poisson arrivals and independent and identically distributed claim sizes.

At a more detailed level, the finitely many major clients are characterized via exponentially distributed inter-arrival times and corresponding claim-size distributions, while  the dynamics of the small clients  are modeled by means of spectrally-positive Lévy processes. We succeeded in characterizing the distribution of the running maximum of the net cumulative claim process, at an exponentially distributed point in time, in terms of its Laplace-Stieltjes transform. Furthermore, for a simplified model in which the small clients are left out, we were able to determine the asymptotics of the ruin probability for both phase-type and regularly varying distributed claims. Finally, for the same model, we also determined the distribution of the overshoot over an exponentially distributed level in terms of its Laplace-Stieltjes transform.

The model analyzed in this paper is an example of a Markov additive process with a non-irreducible background process. A general analysis of extreme values for this type of processes was performed in \cite{KMD}. Due to the fact that in the present paper we consider a {\it specific} Markov additive process with a non-irreducible background process, we were able to find considerably more explicit results. Our findings give rise to the two following interesting follow-up research questions.

\begin{itemize}
    \item[$\circ$] For our specific processes, we found that the Laplace-Stieltjes transform of the running maximum can be evaluated through a straightforward recursion. This raises the question: are there other instances in which similar recursions appear? One other example could be a model in which multiple major clients can arrive at the same time. In this case, the recursive step will likely depend on multiple previous terms (rather than just one).
    \item[$\circ$] In this paper, we have focused on the asymptotics of the ruin probability, within the simplified model in which the small clients correspond to deterministic drifts, considering two types of claim-size distributions. A natural question is: can we extend this analysis to other claim-size distributions? Of particular interest is the class of subexponential distributions, which includes, in addition to distributions with a regularly varying tail, the log-normal and Weibull distributions. Another related question arises: can we determine the asymptotics in models where the spectrally-positive Lévy processes do not correspond to deterministic drifts? In this context, one might consider situations where the small clients are modeled using a compound Poisson process or Brownian motion. However, identifying the asymptotics could prove challenging, as ruin may occur between the arrivals of the major clients.
\end{itemize}

\section*{Software and data availability}
The software used for the preparation of the figures presented in this paper is publicly available at \url{https://github.com/drutgers/Risk-theory-in-a-finite-customer-pool-setting}.

\section*{Acknowledgments}
The authors thank Onno Boxma (Eindhoven University of Technology) and Werner Scheinhardt (Twente University) for helpful discussions. 

\section*{CRediT authorship contribution statement}
\textbf{Michel Mandjes}: Conceptualization, Methodology, Validation, Formal analysis, Writing - original draft, review and editing, Supervision, Project administration, Funding acquisition.

\textbf{Daniël Rutgers}: Conceptualization, Methodology, Validation, Formal analysis, Investigation, Software, Data curation, Writing - original draft, review and editing, Visualization.

\section*{Declaration of competing interest}
There were no competing interests to declare which arose during the preparation of this article.

\section*{Funding Information}
The first author’s research has been funded by the NWO Gravitation project NETWORKS, grant number 024.002.003.

\section*{Declaration of Generative AI and AI-assisted technologies in the writing process}
The authors declare that no generative AI and/or AI-assisted technologies were used in the writing process.

\vb

\end{document}